\documentclass[a4paper,10pt]{article}
\usepackage[latin1]{inputenc}
\usepackage[T1]{fontenc}
\usepackage[francais]{babel}
\usepackage{amsmath}
\usepackage{amsfonts}
\usepackage{amssymb}
\usepackage{amsthm}
\usepackage{mathrsfs}
\usepackage{graphicx}
\theoremstyle{definition}
\newtheorem{definition}{\textbf{Définition}}[section]

\newtheorem{remarque}[definition]{\textbf{Remarque}}

\theoremstyle{plain}

\newtheorem{theoreme}[definition]{\textbf{Théorème}}
\newtheorem{prop}[definition]{\textbf{Proposition}}

\newtheorem{lemme}[definition]{\textit{Lemme}}

 \newcommand{\cRM}[1]{\MakeUppercase{\romannumeral #1}} 

\DeclareMathOperator{\e}{e}
\DeclareMathOperator{\li}{li}

\title{\Huge{\bf{Sur la fonction Delta de~Hooley associée à des caractères }}}
\author{Alexandre Lartaux}

\begin{document}

\maketitle

\vglue0.3cm
\hglue0.02\linewidth\begin{minipage}{0.9\linewidth}
\begin{center}
{Universit\'e de Paris, Sorbonne Université}\\
CNRS,  \\
 Institut de~Math\'ematiques~de Jussieu- Paris Rive Gauche,\\
F-75013 Paris, France \\
 E-mail : \parbox[t]{0.45\linewidth}{\texttt{alexandre.lartaux@imj-prg.fr}}

\end{center}
\end{minipage}
\renewcommand{\abstractname}{Abstract}

\begin{abstract}
Let $(f_1,f_2)$ a $2$-tuple of arithmetic functions and
$$\Delta_3(n,f_1,f_2):=\sup\limits_{\substack{(u_1,u_2) \in \mathbb{R}^{2} \\(v_1,v_2) \in [0,1]^{2}}}\Big\lvert \sum\limits_{\substack{d_1 d_{2} \mid n \\ \e^{u_i}<d_i\leqslant \e^{u_i+v_i}}}{f_1(d_1) f_{2}(d_{2})} \Big\rvert{\rm .}$$
In this paper, we give an upper bound of the second moment of $\Delta_3(n,\chi_1,\chi_2)$ when $\chi_1$ and $\chi_2$ are two non principal Dirichlet characters, following methods developed by La Bretèche and Tenenbaum in \cite{B}. This upper bound is a main step for the asymptotic counting of the number of ideals of norm fixed, which will be developped in \cite{L}.
\end{abstract}

\tableofcontents

\section{Introduction}
\subsection{Notations et énoncés des théorèmes}
Pour étudier l'ensemble
$$\{ \log(d) ; d\mid n \}{\rm ,}$$
la fonction $\Delta$ de Hooley, introduite dans \cite{H} et définie par
\begin{equation}
\label{eq Delta}
\Delta(n)=\max\limits_{u \in \mathbb{R}}\big(\Delta(n,u)\big)
\end{equation}
où
\begin{equation}
\label{eq Deltau}
\Delta(n,u)=\sum\limits_{\substack{d \mid n \\ \e^u<d\leqslant \e^{u+1}}}{1}{\rm .}
\end{equation}
est un outil privilégié. Un résultat remarquable portant sur l'étude de l'ordre moyen de la fonction $\Delta$ a été apporté par Hall et Tenenbaum dans \cite{HT}. Il y est démontré que  cet ordre moyen  est $\log(n)^{o(1)}$. \\

Une généralisation de cette fonction est définie de la manière suivante. \\Pour $r\geqslant 2$ et $n\geqslant 1$, nous notons
\begin{equation}
\label{eq Deltar}
\Delta_r(n)=\max\limits_{\mathbf{u} \in \mathbb{R}^{r-1}}\big(\Delta_r(n,\mathbf{u})\big)
\end{equation}
où, pour $\mathbf{u}=(u_1,\ldots,u_{r-1})$,
\begin{equation}
\label{eq Deltaru}
\Delta_r(n,\mathbf{u})=\sum\limits_{\substack{d_1\ldots d_{r-1}\mid n \\ \e^{u_i}<d_i\leqslant \e^{u_i+1}}}{1}{\rm .}
\end{equation}
Hall et Tenenbaum (\cite{HT}) parviennent à étendre la majoration de l'ordre moyen de $\Delta_r(n)$ lorsque $r\geqslant 3$. Aussi, l'ordre moyen obtenu est $\log(n)^{o(1)}$ pour $r\geqslant 2$.  \\

Dans \cite{B}, La Bretèche et Tenenbaum s'intéressent à une autre généralisation. Ils étudient désormais la fonction Delta associée à un caractère de Dirichlet. Ils définissent donc, pour $n\geqslant 1$ et $\chi$ un caractère de Dirichlet non principal
\begin{equation}
\label{eq Deltaf}
\Delta(n,\chi)=\max\limits_{\substack{u \in \mathbb{R} \\ v \in [0,1]}}|\Delta(n,\chi,u,v)|
\end{equation}
où
\begin{equation}
\label{eq Deltafuv}
\Delta(n,\chi,u,v)=\sum\limits_{\substack{d \mid n \\\e^u<d\leqslant \e^{u+v}}}{\chi(d)}{\rm .}
\end{equation}
Le caractère $\chi$ n'étant pas positif, le maximum en $v$ n'est pas forcément atteint en $v=1$, ce qui justifie l'apparition de cette nouvelle variable. Ils montrent que les compensations dues aux oscillations des caractères sont d'ordre statistique, puisque l'ordre moyen de $\Delta(n,\chi)^2$ est le même que celui de $\Delta(n)$ à un facteur $(\log n)^{o(1)}$ près. \\

Dans cet article, nous nous intéressons à la fonction $\Delta_3$ lorsqu'elle est associée à deux caractères de Dirichlet. Pour $n \in \mathbb{N}^*$, $\boldsymbol{\chi}=(\chi_1,\chi_2)$ un couple de deux caractères de Dirichlet non principaux, $\mathbf{u}=(u_1,u_2)\in \mathbb{R}^2$ et $\mathbf{v}=(v_1,v_2)\in~[0,1]^2$, nous posons
\begin{equation}
\label{eq Delta3fuv}
\Delta_3(n,\boldsymbol{\chi},\mathbf{u},\mathbf{v})=\sum\limits_{\substack{d_1d_2\mid n\\ \e^{u_i}<d_i\leqslant \e^{u_i+v_i}}}{\chi_1(d_1)\chi_2(d_2)}{\rm ,}
\end{equation}

\begin{equation}
\label{eq Delta3f}\Delta_3(n,\boldsymbol{\chi})=\sup\limits_{\substack{ \mathbf{u} \in \mathbb{R}^2\\ \mathbf{v} \in [0,1]^2}} |\Delta_3(n,\boldsymbol{\chi},\mathbf{u},\mathbf{v})| {\rm .}
\end{equation}
Nous introduisons la constante $\rho$ définie par
\begin{equation}
\label{eq lambda}
\begin{split}
\rho:= \frac{1}{2\pi}\int_{-\pi}^{\pi}{\max\big(1,|1+\e^{i\vartheta}|^{2}\big){\rm d}\vartheta}-2=\frac{\sqrt{3}}{\pi}-\frac{1}{3}\approx 0,218 {\rm ,}
\end{split}
\end{equation}
et pour $y>0$, nous notons
\begin{equation}
\label{eq m}
m(y,\rho):=\left\{
 \begin{array}{ll}
(\rho+2)y  & \mbox{si }   y \leqslant \frac{1}{1-\rho} {\rm ,}\\
3y-1 & \mbox{sinon,} \\
\end{array}
\right.
\end{equation}
c'est-à-dire $m(y,\rho)=\max(3y-1,(\rho+2)y)$. \\

Pour la suite, lorsque $\chi_1$ et $\chi_2$ sont deux caractères de Dirichlet, nous notons
\begin{equation}
\label{chi}
\boldsymbol{\chi}:=(\chi_1, \chi_2)
\end{equation}
le couple constitué de ces deux caractères, et nous notons $r_1$ l'ordre de $\chi_1$ et $r_2$ l'ordre de $\chi_2$. Par ailleurs, nous notons également $q_1$ et $q_2$ les conducteurs respectifs de $\chi_1$ et $\chi_2$.\\
Pour $A>0$, $c>0$ et $\eta \in ]0, 1[$, nous disons qu'une fonction positive et multiplicative $g$ appartient à la classe $\mathcal{M}_A(c,\eta)$ si les conditions suivantes sont vérifiées
\begin{align}
  & \forall p \in \mathcal{P} \mbox{ et } \forall \nu \geqslant 1, \; g(p^{\nu}) \leqslant A^{\nu} \nonumber  \\
  & \forall \varepsilon >0 \mbox{ et } \forall n \geqslant 1,\; g(n) \ll_{\varepsilon} n^{\varepsilon} \nonumber \\
  & \sum\limits_{p\leqslant x}{g(p)}= y \li(x)+O(x\e^{-2c(\log x)^{\eta}}) 
  \label{eq def} 
\end{align}
\noindent
pour une certaine constante $y=y(g)>0$. \\

Lorsque $\chi$ est un caractère de Dirichlet d'ordre $r$, nous définissons $\mathcal{M}_A(\chi,c,\eta)$ la classe des fonctions multiplicatives $g$ appartenant à $\mathcal{M}_A(c,\eta)$ et vérifiant
\begin{equation}
\label{eq g}
\forall k \in \mathbb{Z}, \; \sum\limits_{\substack{p\leqslant x \\ \chi(p)=\zeta^k}}{g(p)}=\frac{y}{r} \; \li(x)+ O(x\e^{-2c(\log x)^{\eta}}) 
\end{equation}
où $y=y(g)$ et  
\begin{equation}
\label{eq zeta}
\zeta:=\e^{2i\pi/r} {\rm .}
\end{equation}

Enfin, pour $\boldsymbol{\chi}=(\chi_1, \chi_2)$ un couple de caractères de Dirichlet, $A>0$, $c>0$ et $\eta \in ]0, 1[$, nous définissons la classe $\mathcal{M}_A(\boldsymbol{\chi},c,\eta)$ par
$$\mathcal{M}_A(\boldsymbol{\chi},c,\eta):=\mathcal{M}_A(\chi_1,c,\eta)\bigcap\mathcal{M}_A(\chi_2,c,\eta)\bigcap\mathcal{M}_A(\chi_1\overline{\chi_2},c,\eta){\rm .}$$
Lorsque $x\geqslant 16$, nous notons
$$\mathcal{L}(x):=\e^{\sqrt{\log_2(x)\log_3(x)}} $$
\noindent
où $\log_k$ désigne la $k$-ième itération de la fonction $\log$. Enfin, si $g$ est une fonction arithmétique et $\boldsymbol{\chi}$ un couple de caractères de Dirichlet, nous notons
\begin{equation}
\label{eq theo 1.2}
\mathfrak{S}(x,g,\boldsymbol{\chi}):=\sum\limits_{n \leqslant x}{g(n)\Delta_3(n,\boldsymbol{\chi})^{2}} {\rm .}
\end{equation}

\begin{theoreme}
\label{theo 1}
Soit $\boldsymbol{\chi}$ défini en \eqref{chi} un couple de  caractères de Dirichlet non principaux tel que $\chi_1\overline{\chi_2}$ ne soit pas principal, $A$, $c$ des constantes strictement positives, $\eta \in  ]0,1[$ et $g$ appartenant à $\mathcal{M}_A(\boldsymbol{\chi},c,\eta)$. Pour $y=y(g)>0$, il existe une constante $\alpha >0$, dépendant au plus de $g$, $\boldsymbol{\chi}$, $c$ et $\eta$, telle que nous ayons, uniformément pour $x\geqslant 16$ 
\begin{equation}
\label{eq theo 1.1}
\mathfrak{S}(x,g,\boldsymbol{\chi})\ll x(\log x)^{\max\{y-1,(\rho+2)y-2,3y-3\}}\mathcal{L}(x)^{\alpha}{\rm .}
\end{equation}
En particulier, si $y(g)=1$, alors il existe une constante $\alpha>0$ telle que nous ayons, uniformément pour $x\geqslant 16$
\begin{equation}
\label{eq coro}
\mathfrak{S}(x,g,\boldsymbol{\chi})\ll x(\log x)^{\rho}\mathcal{L}(x)^{\alpha}{\rm ,}
\end{equation}
où $\rho$ est défini en \eqref{eq lambda}.
\end{theoreme}
\begin{remarque}
\label{rq prem}
Comme nous le verrons par la suite, la première étape de la démonstration du Théorème \ref{theo 1} consiste  à ramener le problème à l'étude d'une somme d'entiers sans facteur carré, pondérés par un poids $\frac{1}{n}$. Pour cela, nous utilisons une propriété de sous multiplicativité de la fonction $\Delta$. Cette propriété nous permet de supposer également que $g(n)=0$ si $(n,q_1q_2) \neq 1$ où $q_1$ et $q_2$ désignent les conducteurs respectifs de $\chi_1$ et $\chi_2$.
\end{remarque}
Nous pouvons alors nous demander à quel point le Théorème \ref{theo 1} est optimal. \`A cet effet, nous énonçons le théorème de minoration suivant.
\begin{theoreme}
\label{prop 2}
Soient $y>0$ un réel et $\boldsymbol{\chi}$ un couple de caractères de Dirichlet non principaux. \\

\noindent
(i) Si $\chi_1\overline{\chi_2}$ n'est pas principal, alors nous avons lorsque $x\geqslant 3$,
$$\sum\limits_{n\leqslant x}{\mu^2(n)y^{\omega(n)}\Delta_3^2(n,\boldsymbol{\chi})} \gg x(\log x)^{\max\{y-1,3y-3\}} {\rm .}$$

\noindent
(ii) Si $\chi_1\overline{\chi_2}$ est principal, alors nous avons lorsque $x\geqslant 3$,
$$\sum\limits_{n\leqslant x}{\mu^2(n)y^{\omega(n)}\Delta_3^2(n,\boldsymbol{\chi})} \gg x(\log x)^{\max\{y-1,5y-4\}} {\rm .}$$
\end{theoreme}




Les résultats du Théorème \ref{theo 1} et du Théorème \ref{prop 2} montrent que l'exposant $\max\{y-1,(\rho+2)y-2,3y-3\}$ est optimal lorsque $y\leqslant 1/(1+\rho)$ ou $y\geqslant 1/(1-\rho)$. Il est par ailleurs intéressant de comparer ces résultats avec ceux obtenus par Hall et Tenenbaum. Dans \cite{HT}, les auteurs établissent les estimations suivantes.
\begin{theoreme}[Hall, Tenenbaum \cite{HT}]
Pour tout $0<y<1$, il existe une constante $\alpha$ pour laquelle nous avons, lorsque $x\geqslant 16$,
$$\sum\limits_{n\leqslant x}{y^{\omega(n)}\Delta_3(n)}\ll x(\log x)^{y-1}\mathcal{L}(x)^{\alpha}{\rm .}$$
Pour tout $y\geqslant 1$, il existe une constante $\beta$ pour laquelle nous avons, lorsque $x\geqslant 16$,
$$\sum\limits_{n\leqslant x}{y^{\omega(n)}\Delta_3(n)}\ll x(\log x)^{3y-3}\mathcal{L}(x)^{\beta}{\rm .}$$
Les constantes $\alpha$ et $\beta$ sont explicites.
\end{theoreme} 
Par ailleurs, des minorations triviales montrent que les exposants de $\log x$ sont optimaux. Nous pouvons donc observer que pour $y\leqslant 1/(1+\rho)$ et pour $y\geqslant 1/(1-\rho)$, $y^{\omega(n)}\Delta_3(n,\chi_1,\chi_2)^2$ se comporte en moyenne comme $y^{\omega(n)}\Delta_3(n)$. Néanmoins pour les autres valeurs de $y$, notamment pour $y=1$, nous ne savons pas si cela reste vrai. \\

Pour démontrer le Théorème \ref{theo 1}, nous étudierons les quantités suivantes
\begin{equation}
\label{eq C}
C(n,\boldsymbol{\chi}):=\sup\limits_{\substack{u_2 \in \mathbb{R} \\ v_1,v_2 \in [0,1]^2}} C(n,\boldsymbol{\chi},u_2,v_1,v_2)
\end{equation}
et
\begin{equation}
\label{eq D}
D(n,\boldsymbol{\chi}):=\sup\limits_{\substack{u \in \mathbb{R} \\ v_1,v_2 \in [0,1]^2}} D(n,\boldsymbol{\chi},u,v_1,v_2)
\end{equation}
\noindent
où
\begin{equation}
\label{eq Cu}
C(n,\boldsymbol{\chi},u_2,\mathbf{v}):=\int_{\mathbb{R}}{|\Delta_3(n,\boldsymbol{\chi},u_1,u_2,\mathbf{v})|^2{\rm d}u_1}
\end{equation}
et
\begin{equation}
\label{eq Du}
D(n,\boldsymbol{\chi},u,\mathbf{v}):=\int_{\mathbb{R}}{|\Delta_3(n,\boldsymbol{\chi},u-u_1,u_1,\mathbf{v})|^2{\rm d}u_1}
\end{equation}
\noindent
avec $\mathbf{v}=(v_1,v_2)$,  $\boldsymbol{\chi}$ un couple de caractères de Dirichlet et $n \in \mathbb{N}^*$.

Comme nous le verrons par la suite, les estimations de $C(n,\boldsymbol{\chi})$ et de $D(n,\boldsymbol{\chi})$ passent par l'estimation de 
$$\Delta^{(k)}(n,\boldsymbol{\chi},\vartheta):=\sup\limits_{(u,v_1,v_2)\in \mathbb{R}\times [0,1]^2} \big\lvert \Delta^{(k)}(n,\boldsymbol{\chi},\vartheta,u,v_1,v_2)\big\rvert{\rm ,}$$
où
\begin{equation}
\label{Dktheta}
\Delta^{(k)}(n,\boldsymbol{\chi},\vartheta,u,v_1,v_2):=\sum\limits_{\substack{d_1d_2\mid n \\ \e^u<d_1\leqslant \e^{u+v_1}}}{\chi_1(d_1)\chi_2(d_2)d_2^{i\vartheta}(u+v_1-v_2-\log d_1)^k}{\rm .}
\end{equation} 
Toutes ces quantités sont étudiées dans la partie 3.

\subsection{Applications}

Les résultats que nous obtenons pour l'ordre moyen de $\Delta_3(n,\chi_1,\chi_2)^2$ ne fournissent pas des compensations statistiques qui nous permettraient de le comparer à l'ordre moyen de $\Delta_3(n)$. Le Théorème \ref{theo 1} est néanmoins utilisé de manière cruciale dans l'article \cite{L} afin de fournir, dans un cadre particulier, une estimation asymptotique en $\xi$ de la quantité suivante.
$$Q(F,\xi,\mathcal{R}):=\sum\limits_{\mathbf{x} \in \mathbb{Z}^2\cap \mathcal{R}(\xi)}{r_3(F(\mathbf{x}))}{\rm ,}$$
où $F$ est une forme binaire de degré $3$, $r_3(n)$ compte le nombre d'idéaux d'un anneau d'entiers de corps de nombre de norme $n$, $\mathcal{R}$ est un domaine de $\mathbb{R}^2$. Enfin, nous désignons par $\mathcal{R}(\xi)$ le domaine de $\mathbb{R}^2$ obtenu en dilatant $\mathcal{R}$ par $\xi$, c'est-à-dire
$$\mathcal{R}(\xi):=\{\mathbf{x} \in \mathbb{R}^2 : \mathbf{x}/\xi \in \mathcal{R}\}{\rm .}$$

 Dans le cas où le corps de nombre $\mathbb{K}$ est une extension cyclique de $\mathbb{Q}$ de degré $3$, en notant $G:=\rm{Gal}(\mathbb{K}/\mathbb{Q})$ son groupe de Galois et $\chi$  un caractère non principal de $G$, nous pouvons relier la fonction $r_3$ au caractère $\chi$ par l'égalité suivante.
 $$r_3(n)=(1*\chi*\chi^2)(n){\rm ,}$$
 où $*$ désigne le produit de convolution entre deux fonctions arithmétiques. Pour $f$ et $g$ deux telles fonctions , celui-ci est la fonction arithmétique définie par
$$(f*g)(n)=\sum\limits_{d\mid n}{f(d)g\Big(\frac{n}{d}\Big)} \;\;\;\;\; \forall n\geqslant 1{\rm .}$$
Nous formulons quatre hypothèses sur le domaine $\mathcal{R}$ et sur la forme binaire $F$, valables pour certaines valeurs $\sigma>0$ et $\vartheta>0$.  
\begin{align*}
\text{(H1)} & &&\text{Le domaine } \mathcal{R} \text{ est un ouvert borné convexe dont } \\
& &&\text{la frontière est continuement différentiable.} \\
\text{(H2)} & &&\forall \mathbf{x} \in \mathcal{R}\mbox{, } ||\mathbf{x}||\leqslant \sigma \\
\text{(H3)} & && \forall \mathbf{x} \in \mathcal{R}\mbox{, } |F(\mathbf{x})|\leqslant \vartheta^3{\rm .}\\
\text{(H4)} & && \text{La forme } F \text{ est irréductible sur } \mathbb{K}{\rm .}
\end{align*}

\begin{theoreme}
Soient $\mathbb{K}$ une extension cyclique de $\mathbb{Q}$ de degré $3$, $\varepsilon>0$, $\sigma>0$, $\xi>0$, $F$ une forme binaire de degré $3$ et $\mathcal{R}$ un domaine de $\mathbb{R}^2$ tels que les hypothèses (H1), (H2), (H3) et (H4) soient vérifiées. Sous les conditions
$$1/\sqrt{\xi}\leqslant \sigma \leqslant \xi^{3/2}{\rm ,}\;\;\;\;\; 1/\sqrt{\xi}\leqslant \vartheta \leqslant \xi^{3/2}{\rm ,}$$
nous avons
\begin{equation}
\label{1}
Q(F,\xi,\mathcal{R})=K(F){\rm vol}(\mathcal{R})\xi^2+O\Big(\frac{||F||^{\varepsilon}(\sigma^2+\vartheta^2)\xi^2}{(\log \xi)^{0,0034}}\Big){\rm .}
\end{equation}
où $||F||$ désigne le maximum des coefficients de la forme $F$ et $K(F)$ est une constante explicite possédant une interprétation géométrique.
\end{theoreme}

\section{Lemmes préliminaires}
\subsection{Majorer une norme infinie par une norme $L^q$}
La démonstration du Théorème \ref{theo 1} suit la méthode différentielle de \cite{B}. Les lemmes préliminaires que nous énonçons ici sont analogues à ceux de notre article de référence. Cependant, comme il s'agit ici essentiellement d'une généralisation à deux variables de ces lemmes, plusieurs difficultés apparaissent, nous rédigeons donc en détails les démonstrations.

Avant d'énoncer le premier lemme, quelques notations sont nécessaires. \\

Lorsque $n \geqslant 1$, nous posons
\begin{align}
\label{E}
E(n):=\min\limits_{\substack{dd'\mid n \\ d<d'}}{\log \frac{d'}{d}}\mbox{,} & &&
E^*(n):=\min{\big(1,E(n)\big)} {\rm .}
\end{align}
Pour $n \geqslant 1$, $q \geqslant 1$, $\mathbf{f}$ un couple de deux fonctions arithmétiques, nous posons
\begin{equation}
\label{eq M}
M_{2q}(n,\mathbf{f}):=\int_{[0,1]^2}{\int_{\mathbb{R}^2}{|\Delta_3(n,\mathbf{f},\mathbf{u},\mathbf{v})|^{2q}{\rm d} \mathbf{u}}{\rm d} \mathbf{v}} {\rm .}
\end{equation}
Pour $\mathbf{x} \in \mathbb{R}^2$, nous  notons
$$ ||\mathbf{x}||_{\infty}=\max(|x_1|,|x_2|) {\rm .}$$
Lorsque $\chi$ est un caractère de Dirichlet, nous notons
\begin{equation}
\label{eq lemme 1.2}
\Delta^*(n,\chi):=\max\limits_{d \mid n}\big(\Delta(d,\chi)\big){\rm ,}
\end{equation}
où $\Delta(n,\chi)$ est défini en \eqref{eq Deltaf}.
\begin{lemme}
\label{lemme 1}
Pour $n\geqslant 1$, $q\geqslant 1$, $\boldsymbol{\chi}$ défini en \eqref{chi}, un couple de deux caractères de Dirichlet, nous avons
\begin{equation}
\label{eq lemme 1}
\Delta_3(n,\boldsymbol{\chi})^2\leqslant 4E^*(n)^{-4/q}M_{2q}(n,\boldsymbol{\chi})^{1/q}+32\big(\Delta^*(n,\chi_1)^2+\Delta^*(n,\chi_2)^2\big) {\rm .}
\end{equation}
\end{lemme}

Les deuxième et troisième membres qui apparaissent dans ce majorant ne peuvent être trivialement traités par le théorème 1.1 de \cite{B} puisque le maximum sur les diviseurs de $n$ n'est pas forcément atteint en $n$, contrairement au cas où $\chi_1=\chi_2=1$. Il nécessite donc une résolution particulière que nous développons dans la sous-section $4.1$. Cet aspect constitue une des innovations de ce travail.

\begin{proof}

Si $\Delta_3(n,\boldsymbol{\chi})\leqslant 4\big(\Delta^*(n,\chi_1)+\Delta^*(n,\chi_2)\big)$, le résultat est évident. \\
Sinon, nous notons $\mathbf{u_0}$ et $\mathbf{v_0}$ les réels tels que
$$\Delta_3(n,\boldsymbol{\chi})=|\Delta_3(n,\boldsymbol{\chi},\mathbf{u_0},\mathbf{v_0})| $$
\noindent
avec $\mathbf{u_0}=(u_{01},u_{02})$ et $\mathbf{v_0}=(v_{01},v_{02})$.
\noindent
Pour tout $\mathbf{u}=(u_1,u_2)$, $\mathbf{v}=(v_1,v_2)$ tels que 
$$||\mathbf{u}+\mathbf{v}-\mathbf{u_0}-\mathbf{v_0}||_{\infty}\leqslant E^*(n){\rm ,}$$ 
nous avons 
\begin{align*}
\Delta_3(n,\boldsymbol{\chi},\mathbf{u},\mathbf{v}) & =&& \!\!\!\!\!\!\!\!\!\!\!\!\!\!\!\!\!\!\!\!\!\!\!\!\!&&\!\!\!\!\!\!\!\!\!\!\!\!\!\!\!\!\!\!\!\!\sum\limits_{\substack{\log(d) \in ]u_1,u_1+v_1]\\d\mid n}}{\chi_1(d)\Delta\Big(\frac{n}{d},\chi_2,u_2,v_2\Big)} \\
 & =&& \!\!\!\!\!\!\!\!\!\!\!\!\!\!\!\!\!\!\!\!\!\!\!\!\!&&\!\!\!\!\!\!\!\!\!\!\!\!\!\!\!\!\!\!\!\!\!\!\!\!\!\sum\limits_{\substack{\log(d) \in ]u_{01},u_{01}+v_{01}] \\d\mid n}}{\chi_1(d)\Delta\Big(\frac{n}{d},\chi_2,u_2,v_2\Big)} \\
 &&&\!\!\!\!\!\!\!\!\!\!\!\!\!\!\!+&&\!\!\!\!\!\!\!\!\!\!\!\!\!\!\!\!\!\!\!\!\!\!\!\!\sum\limits_{\substack{\log(d) \in ]u_1,u_1+v_1]\\ \log(d)\notin ]u_{01},u_{01}+v_{01}]\\ d\mid n}}{\chi_1(d)\Delta\Big(\frac{n}{d},\chi_2,u_2,v_2\Big)}\\
 &&&\!\!\!\!\!\!\!\!\!\!\!\!\!\!\!-&&\!\!\!\!\!\!\!\!\!\!\!\!\!\!\!\!\!\!\!\!\!\!\!\!\!\sum\limits_{\substack{\log(d) \in ]u_{01},u_{01}+v_{01}] \\ \log(d)\notin ]u_1,u_1+v_1] \\ d\mid n}}{\chi_1(d)\Delta\Big(\frac{n}{d},\chi_2,u_2,v_2\Big)} {\rm .}
\end{align*}
Les deux dernières sommes comportent au plus un terme, elles sont donc majorées en valeur absolue par
$\Delta^*(n,\chi_2) {\rm .}$
Nous pouvons alors effectuer le même travail avec la première somme, ce qui nous permet de retrouver $\Delta_3(n,\boldsymbol{\chi},\mathbf{u_0},\mathbf{v_0})$ et un reste qui ne dépassera pas $2\Delta^*(n,\chi_1) {\rm .}$
Ainsi, pour tout $\mathbf{u}$, $\mathbf{v}$ tels que $||\mathbf{u}+\mathbf{v}-\mathbf{u_0}-\mathbf{v_0}||_{\infty}\leqslant E^*(n)$, nous avons
$$|\Delta_3(n,\boldsymbol{\chi},\mathbf{u},\mathbf{v})|\geqslant \Delta_3(n,\boldsymbol{\chi},\mathbf{u_0},\mathbf{v_0})-2\big(\Delta^*(n,\chi_1)+\Delta^*(n,\chi_2)\big) {\rm .}$$
Puisque $\Delta_3(n,\boldsymbol{\chi},\mathbf{u_0},\mathbf{v_0})\geqslant 4\big(\Delta^*(n,\chi_1)+\Delta^*(n,\chi_2)\big)$, nous avons
$$|\Delta_3(n,\boldsymbol{\chi},\mathbf{u},\mathbf{v})|\geqslant \frac{1}{2}\Delta_3(n,\boldsymbol{\chi}) {\rm .}$$
\noindent
En élevant cette inégalité à la puissance $2q$, en intégrant cette relation sur l'ensemble de mesure $E^*(n)^4$ sur lequel cette inégalité est valable, puis en prenant la racine $q$-ième, nous obtenons le résultat annoncé. 
\end{proof}
Pour la suite, lorsque $\chi$ est un caractère de Dirichlet, nous notons
\begin{equation}
\label{eq lemme 2.2}
M^{\dagger}_{2q}(n,\chi):=\sum\limits_{d \mid n}{M_{2q}(d,\chi)} 
\end{equation}
avec
\begin{equation}
\label{eq lemme 2.2 bis}
M_{2q}(n,\chi):=\int_{0}^{1}{\int_{\mathbb{R}}{|\Delta(n,\chi,u,v)|^{2q}{\rm d}u}{\rm d}v}{\rm .}
\end{equation}

\begin{lemme}
\label{lemme 2}
Pour $n\geqslant 1$, $q\geqslant 1$, $\chi$ un caractère de Dirichlet, nous avons
\begin{equation}
\label{eq lemme 2.1}
\Delta^*(n,\chi)^2\leqslant 4E^*(n)^{-2/q}M^{\dagger}_{2q}(n,\chi)^{1/q}+4\tau(n)^{1/q} {\rm .}
\end{equation}

\end{lemme}

\begin{proof}
Observons tout d'abord que pour tout $n$ et pour tout $\chi$, nous avons 
$$\Delta^*(n,\chi)^2\leqslant \Big(\sum\limits_{d\mid n}{\Delta(d,\chi)^{2q}}\Big)^{1/q} {\rm .}$$
\noindent
Nous utilisons ensuite le lemme $2.1$ de \cite{B} sous la forme suivante
$$\Delta(d,\chi)^{2q}\leqslant 4^qE^*(d)^{-2}M_{2q}(d,\chi)+4^q$$
\noindent
pour obtenir
\begin{align*}
\Delta^*(n,\chi)^2& \leqslant \Big(\sum\limits_{d\mid n}{4^qE^*(d)^{-2}M_{2q}(d,\chi)+4^q}\Big)^{1/q}\\
&\leqslant \Big(\sum\limits_{d\mid n}{4^qE^*(n)^{-2}M_{2q}(d,\chi)}+4^q\tau(n)\Big)^{1/q}{\rm .}
\end{align*}
Pour conclure, nous utilisons l'inégalité suivante valable pour tout $x,y \geqslant 0$ et $q\geqslant 1$
$$(x+y)^{1/q}\leqslant x^{1/q}+y^{1/q} {\rm .}$$
\end{proof}
Pour les besoins du prochain lemme, nous notons, pour $n \in \mathbb{N}^*$, $\chi$ un caractère de Dirichlet et $\vartheta \in \mathbb{R}$,
\begin{equation}
\label{eq tautheta}
\tau(n,\chi,\vartheta):=\sum\limits_{d\mid n}{\chi(d)d^{i\vartheta}}{\rm .}
\end{equation}
Pour $\boldsymbol{\chi}$ un couple de caractères de Dirichlet, $q \in \mathbb{N}$ et $j\leqslant q$, nous notons
\begin{equation}
\label{eq Mijtheta}
M_{2q}^{(j)}(n,\boldsymbol{\chi},\vartheta):=\int_{[0,1]^2}{\int_{\mathbb{R}}{|\Delta^{(j)}(n,\boldsymbol{\chi},\vartheta,u,v_1,v_2)|^{2q}{\rm d}u}{\rm d}v_1{\rm d}v_2} {\rm ,}
\end{equation}
où les quantités $\Delta^{(j)}(n,\boldsymbol{\chi},\vartheta,u,v_1,v_2)$ sont définies en \eqref{Dktheta}. Enfin, pour $k \in \mathbb{N}$, nous posons
$$E_k(n):=\min\Big(\frac{\log (3/2)}{k+1},E(n)\Big){\rm .}$$
\begin{lemme}
\label{lemme 3}
Pour $k\in \mathbb{N}$, $n\geqslant 1$, $q\geqslant 1$, $\boldsymbol{\chi}$ un couple de caractères de Dirichlet et $\vartheta \in \mathbb{R}$, nous avons
$$\Delta^{(k)}(n,\boldsymbol{\chi},\vartheta)^2\leqslant 16E_k(n)^{-3/q}\sum\limits_{j\leqslant k}{M_{2q}^{(j)}(n,\boldsymbol{\chi},\vartheta)^{1/q}}+64\e^2\max\limits_{d \mid n}\big(|\tau(d,\chi_2,\vartheta)|^2\big) {\rm .}$$
\noindent

\end{lemme}

\begin{proof}
Pour tout $k \in \mathbb{N}$, nous notons $j_k$ un entier $j\leqslant k$ tel que 
$$\Delta^{(j_k)}(n,\boldsymbol{\chi},\vartheta):=\max\limits_{j\leqslant k} \Delta^{(j)}(n,\boldsymbol{\chi},\vartheta){\rm .}$$
Si $\Delta^{(j_k)}(n,\boldsymbol{\chi},\vartheta)\leqslant 8\e\max\limits_{d \mid n}\big(|\tau(d,\chi_2,\vartheta)|^2\big)$, le résultat est évident. \\
Sinon, nous notons $u_0$, $v_{01}$ et $v_{02}$ les réels tels que
$$\Delta^{(j_k)}(n,\boldsymbol{\chi},\vartheta)=|\Delta^{(j_k)}(n,\boldsymbol{\chi},\vartheta,u_0,v_{01},v_{02})|{\rm .} $$
\pagebreak

\noindent
Pour tout $u,v_1,v_2$ tels que 
$$\lvert u-u_0\rvert +\lvert v_1-v_{01}\rvert +\lvert v_2-v_{02}\rvert\leqslant E_k(n){\rm ,}$$ 
nous avons 
\begin{align*}
\Delta^{(j_k)}(n,\boldsymbol{\chi},\vartheta,u,v_1,v_2) & =&& &&\!\!\!\!\!\sum\limits_{\substack{\log(d) \in ]u,u+v_1]\\d\mid n}}{\chi_1(d)\tau\Big(\frac{n}{d},\chi_2,\vartheta\Big)(u+v_1-v_2-\log d)^{j_k}} \\
 & =&& &&\!\!\!\!\!\!\!\!\!\!\sum\limits_{\substack{\log(d) \in ]u_0,u_0+v_{01}] \\d\mid n}}{\chi_1(d)\tau\Big(\frac{n}{d},\chi_2,\vartheta\Big)(u+v_1-v_2-\log d)^{j_k}} \\
 &&&+&&\!\!\!\!\!\!\!\!\!\sum\limits_{\substack{\log(d) \in ]u,u+v_1]\\ \log(d)\notin ]u_{0},u_{0}+v_{01}]\\ d\mid n}}{\chi_1(d)\tau\Big(\frac{n}{d},\chi_2,\vartheta\Big)(u+v_1-v_2-\log d)^{j_k}}\\
 &&&-&&\!\!\!\!\!\!\!\!\!\!\sum\limits_{\substack{\log(d) \in ]u_{0},u_{0}+v_{01}] \\ \log(d)\notin ]u,u+v_1] \\ d\mid n}}{\chi_1(d)\tau\Big(\frac{n}{d},\chi_2,\vartheta\Big)(u+v_1-v_2-\log d)^{j_k}} {\rm .}
\end{align*}
Les deux dernières sommes comportent au plus un terme. Par ailleurs, nous avons, pour tout $\log(d) \in ]u_{0},u_{0}+v_{01}] \cup ]u,u+v_1]$
$$(u+v_1-v_2-\log d)^{j_k}\leqslant (1+\frac{1}{k+1})^{j_k}\leqslant \e {\rm .}$$
Ainsi, ces deux sommes sont majorées par 
$$\e \max\limits_{d \mid n}\big(|\tau(d,\chi_2,\vartheta)|\big){\rm .}$$
Pour la première somme, nous écrivons 
\begin{align*}
&\sum\limits_{\substack{\log(d) \in ]u_0,u_0+v_{01}] \\d\mid n}}{\chi_1(d)\tau\Big(\frac{n}{d},\chi_2,\vartheta\Big)(u+v_1-v_2-\log d)^{j_k}}\\
=&\!\!\!\!\!\!\!\!\sum\limits_{\substack{\log(d) \in ]u_0,u_0+v_{01}] \\d\mid n}}{\!\!\!\!\!\!\!\chi_1(d)\tau\Big(\frac{n}{d},\chi_2,\vartheta\Big)\big(u_0+v_{01}-v_{02}-\log d+(u-u_0+v_1-v_{01}+v_{02}-v_2)\big)^{j_k}}\\
=&\sum\limits_{j\leqslant j_k}{\binom{j_k}{j}(u-u_0+v_1-v_{01}+v_{02}-v_2)^{j_k-j}}\\
&\sum\limits_{\substack{\log(d) \in ]u_0,u_0+v_{0}] \\d\mid n}}{\chi_1(d)\tau\Big(\frac{n}{d},\chi_2,\vartheta\Big)(u_0+v_{01}-v_{02}-\log d)^{j}}{\rm .}
\end{align*} 
Nous avons ainsi
\begin{align*}
&\Big\lvert\sum\limits_{\substack{\log(d) \in ]u_0,u_0+v_{01}] \\d\mid n}}{\chi_1(d)\tau\Big(\frac{n}{d},\chi_2,\vartheta_2\Big)(u+v_1-v_2-\log d)^{j_k}}\Big\rvert\\
\geqslant & \Delta^{(j_k)}(n,\boldsymbol{\chi},\vartheta)-\big((1+\lvert u-u_0+v_1-v_{01}+v_{02}-v_2\rvert)^{j_k}-1\big)\Delta^{(j_k)}(n,\boldsymbol{\chi},\vartheta){\rm .}
\end{align*}
Par hypothèse, nous avons
$$\big((1+\lvert u-u_0+v_1-v_{01}+v_{02}-v_2\rvert)^{j_k}-1\big)\leqslant \frac{1}{2}{\rm ,}$$
ce qui fournit
\begin{align*}
&\Big\lvert\sum\limits_{\substack{\log(d) \in ]u_0,u_0+v_{01}] \\d\mid n}}{\chi_1(d)\tau\Big(\frac{n}{d},\chi_2,\vartheta_2\Big)(u+v_1-v_2-\log d)^{j_k}}\Big\rvert\\
\geqslant & \frac{1}{2}\Delta^{(j_k)}(n,\boldsymbol{\chi},\vartheta){\rm .}
\end{align*} 
Puisque $\Delta^{(j_k)}(n,\boldsymbol{\chi},\vartheta)\geqslant 8\e\max\limits_{d\mid n}\lvert \tau(d,\boldsymbol{\chi},\vartheta)\rvert$, nous avons
$$|\Delta^{(j_k)}(n,\boldsymbol{\chi},\vartheta,u,v_1,v_2)|\geqslant \frac{1}{4}\Delta^{(j_k)}(n,\boldsymbol{\chi},\vartheta) {\rm .}$$
\noindent
En élevant cette inégalité à la puissance $2q$, en intégrant cette relation sur l'ensemble de mesure $E_k(n)^3$ sur lequel cette inégalité est valable, puis en prenant la racine $q$-ième, nous obtenons

\begin{align*}
\Delta^{(j_k)}(n,\boldsymbol{\chi},\vartheta)^2&\leqslant 16E_k(n)^{-3/q}M_{2q}^{(j_k)}(n,\boldsymbol{\chi},\vartheta)^{1/q}+64\e^2\max\limits_{d \mid n}\big(|\tau(d,\chi_2,\vartheta)|^2\big)\\
& \leqslant 16E_k(n)^{-3/q}\sum\limits_{j\leqslant k}{M_{2q}^{(j)}(n,\boldsymbol{\chi},\vartheta)^{1/q}}+64\e^2\max\limits_{d \mid n}\big(|\tau(d,\chi_2,\vartheta)|^2\big){\rm .}
\end{align*}
L'inégalité $\Delta^{(k)}(n,\boldsymbol{\chi},\vartheta)\leqslant \Delta^{(j_k)}(n,\boldsymbol{\chi},\vartheta)$ nous permet de conclure. 
\end{proof}
\subsection{Un théorème de Shiu et une formule de Parseval}
Le lemme suivant est une adaptation du lemme $2.2$ de \cite{B} qui nous permet de négliger la contribution des entiers $n$ pour lesquels le facteur $E(n)$, défini en \eqref{E}, est trop petit.

\begin{lemme}
\label{lemme 4}
Soient $A$, $c$ des constantes strictement positives, $\eta \in ]0,1[$, $g$ une fonction de $\mathcal{M}_A(c,\eta)$ et $Y\geqslant 27y$ où $y=y(g)>0$. Nous avons uniformément en $\sigma>0$ 
\begin{equation}
\label{eq 4.2}
\sum\limits_{\substack{n \geqslant 1 \\ E(n) \leqslant \sigma^{Y}}}{\frac{\mu^2(n)g(n)\tau_3(n)^{2}}{n^{1+\sigma}}} \ll 1 {\rm .}
\end{equation}
\end{lemme}

\begin{proof}
La preuve est analogue à celle du lemme $2.2$ de \cite{B}. Pour un tel choix de $Y$, dire qu'un entier $n$ vérifie l'inégalité
$$E(n)\leqslant \sigma^Y$$
implique qu'il possède deux diviseurs premiers entre eux $d_1,d_2$ vérifiant d'une part $d_1<d_2\leqslant (1+\sigma^Y)d_1$, et d'autre part $\frac{1}{2\sigma^Y}\leqslant d_1{\rm .}$ 
\goodbreak

\noindent
Ainsi, nous avons
\begin{align*}
\sum\limits_{\substack{n\geqslant 1 \\E(n)\leqslant \sigma^Y}}{\frac{\mu^2(n)g(n)\tau_3(n)^2}{n^{1+\sigma}}} & \leqslant\!\!\!\!\!\!\!\!\!\!\!\!\!\!\sum\limits_{\substack{m,d_1,d_2 \\ \frac{1}{2\sigma^Y}\leqslant d_1<d_2\leqslant (1+\sigma^Y)d_1 \\ (d_1,d_2)=1}}{\!\!\!\!\!\!\!\!\!\!\!\!\!\!\frac{\tau_3(m)^{2}\tau_3(d_1)^{2}\tau_3(d_2)^{2}}{(md_1d_2)^{1+\sigma}}g(m)g(d_1)g(d_2)} \\
& \ll \frac{1}{\sigma^{9y}}\sum\limits_{\frac{1}{2\sigma^Y}\leqslant d_1}{\frac{g(d_1)\tau_3(d_1)^{2}}{d_1^{2+\sigma}}\!\!\!\!\!\!\!\!\!\!\!\!\!\!\sum\limits_{\substack{d_1<d_2\leqslant (1+\sigma^{Y-y})d_1  \\ (d_1,d_2)=1}}{\!\!\!\!\!\!\!\!\!\!\!\!\!\!g(d_2)\tau_3(d_2)^{2}}}{\rm .} 
\end{align*}
\noindent
Un théorème de Shiu \cite{S} nous permet de majorer la somme intérieure, ce qui fournit l'inégalité
\begin{align*}
\sum\limits_{\substack{n\geqslant 1 \\E(n)\leqslant \sigma^Y}}{\frac{\mu^2(n)g(n)\tau_3(n)^2}{n^{1+\sigma}}} & \ll \sigma^{Y-10y}\sum\limits_{d}{\frac{g(d)\tau_3(d)^{2}}{d^{1+\sigma}}\log(d)^{8y}} \\
& \ll \sigma^{Y-27y} \\
& \ll 1 {\rm .}
\end{align*}
\end{proof}
Pour $\boldsymbol{\chi}=(\chi_1,\chi_2)$ un couple de caractères de Dirichlet, $n \in \mathbb{N}^*$ et $(\vartheta_1,\vartheta_2) \in~\mathbb{R}^2$, nous posons
\begin{equation}
\label{tautheta12}
\tau(n,\boldsymbol{\chi},\vartheta_1,\vartheta_2):=\sum\limits_{d_1d_2 \mid n}{\chi_1(d_1)d_1^{i\vartheta_1}\chi_2(d_2)d_2^{i\vartheta_2}}{\rm ,}
\end{equation}
et pour $\vartheta \in \mathbb{R}$, nous posons
\begin{equation}
\label{hattau1}
\widehat{\tau_{1}}(n,\boldsymbol{\chi},\vartheta):=\int_{\mathbb{R}^2}{\big\lvert\tau(n,\boldsymbol{\chi},\vartheta_1,\vartheta)\big\rvert^2\frac{{\rm d}\vartheta_1}{1+\vartheta_1^2}}{\rm .}
\end{equation}
\begin{lemme}
\label{lemme tau1}
Soient $\chi_1$ et $\chi_2$ deux caractères de Dirichlet et $n\geqslant 1$ un entier. Nous avons, uniformément pour $k \in \mathbb{N}$ et $v_1,v_2 \in [0,1]^2$
$$\int_{\mathbb{R}}{|\Delta^{(k)}(n,\boldsymbol{\chi},\vartheta,u,v_1,v_2)|^{2}{\rm d}u}\ll (k+1)^2\widehat{\tau_{1}}(n,\boldsymbol{\chi},\vartheta)$$
où $\Delta^{(k)}(n,\boldsymbol{\chi},\vartheta,u,v_1,v_2)$ est défini en \eqref{Dktheta}.
\end{lemme}

\begin{proof}
Posons
$$f:u\mapsto \Delta^{(k)}(n,\boldsymbol{\chi},\vartheta,u,v_1,v_2){\rm .}$$
Il s'agit d'une somme de fonctions portes, nous pouvons donc calculer sa transformée de Fourier, notée $\widehat{f}$.
\begin{align*}
\widehat{f}(\vartheta_1)&=\int_{\mathbb{R}}{f(u)\e^{-iu\vartheta_1}{\rm d}u}\\
&=\sum\limits_{d_1d_2\mid(n)}{\chi_1(d_1)\chi_2(d_2)d_2^{i\vartheta}\int_{\log d_1-v_1}^{\log d_1}{(u+v_1-v_2-\log d_1)^k \e^{-iu\vartheta_1}{\rm d}u}}\\
&=\sum\limits_{d_1d_2\mid(n)}{\chi_1(d_1)\chi_2(d_2)d_2^{i\vartheta}d_1^{-i\vartheta_1}}\int_{-v_1}^{0}{(u+v_1-v_2)^k \e^{-iu\vartheta_1}{\rm d}u}\\
&=\tau(n,\boldsymbol{\chi},\vartheta,-\vartheta_1)\int_{-v_1}^{0}{(u+v_1-v_2)^k \e^{-iu\vartheta_1}{\rm d}u}{\rm .}
\end{align*}
Or, pour tout $v_1,v_2 \in [0,1]^2$ et $k \in \mathbb{N}$, nous avons
$$\Big\lvert\int_{-v_1}^{0}{(u+v_1-v_2)^k \e^{-iu\vartheta_1}{\rm d}u}\Big\rvert \leqslant \min \Big(1,\frac{k+2}{\lvert \vartheta_1 \rvert}\Big){\rm .}$$
L'inégalité
$$\Big\lvert\int_{-v_1}^{0}{(u+v_1-v_2)^k \e^{-iu\vartheta_1}{\rm d}u}\Big\rvert \leqslant 1$$
est triviale puisqu'on intègre sur un intervalle de taille plus petite que $1$ une fonction majorée par $1$ en valeur absolue.
L'inégalité
$$\Big\lvert\int_{-v_1}^{0}{(u+v_1-v_2)^k \e^{-iu\vartheta_1}{\rm d}u}\Big\rvert \leqslant \frac{k+2}{\lvert \vartheta_1 \rvert}$$
est fournie par une intégration par partie. La formule de Parseval nous permet de conclure.
\end{proof}
\subsection{Des inégalités de Hölder}
Avant d'établir le prochain lemme, quelques notations sont nécessaires.\\
 Pour $n\geqslant 1$, $1\leqslant j \leqslant q$ des entiers, $w \in \mathbb{R}$ et $\boldsymbol{\chi}$ défini en \eqref{chi}, un couple de  caractères de Dirichlet, nous posons
\begin{equation}
\label{N1jq}
N_{1,j,q}(n,\boldsymbol{\chi},w):=\int_{[0,1]^2}{\int_{\mathbb{R}^2}{\Delta_3(n,\boldsymbol{\chi},\mathbf{u},\mathbf{v})^{2j}\Delta_3(n,\boldsymbol{\chi},u_1-w,u_2,\mathbf{v})^{2(q-j)}{\rm d}\mathbf{u}}{\rm d}\mathbf{v}} {\rm ,}
\end{equation}
\begin{equation}
\label{N2jq}
N_{2,j,q}(n,\boldsymbol{\chi},w):=\int_{[0,1]^2}{\int_{\mathbb{R}^2}{\Delta_3(n,\boldsymbol{\chi},\mathbf{u},\mathbf{v})^{2j}\Delta_3(n,\boldsymbol{\chi},u_1-w,u_2+w,\mathbf{v})^{2(q-j)}{\rm d}\mathbf{u}}{\rm d}\mathbf{v}} {\rm .}
\end{equation}
\noindent
De plus, nous posons
\begin{equation}
\label{eq tau2}
\widehat{\tau_{2}}(n,\boldsymbol{\chi}):=\int_{\mathbb{R}^2}{|\tau(n,\boldsymbol{\chi},\boldsymbol{\vartheta})|^2\prod\limits_{i=1}^2{\frac{1}{1+\vartheta_i^2}}{\rm d}\boldsymbol{\vartheta}} {\rm ,}
\end{equation}
où $\tau(n,\boldsymbol{\chi},\boldsymbol{\vartheta})$ est défini en \eqref{tautheta12}.
\begin{lemme}
\label{lemme 5}
Soient $A>0$, $c>0$ des constantes, $\eta \in ]0,1[$  et $\boldsymbol{\chi}$ un couple de caractères de Dirichlet. Pour toute fonction $g$ de $\mathcal{M}_A(c,\eta)$, $q\geqslant 1$, $1\leqslant j \leqslant q-1$, $n\geqslant 1$ et $x\geqslant 2$, nous avons 
\begin{align*}
\sum\limits_{p>x}{\frac{g(p)\log p}{p}N_{1,j,q}(n,\boldsymbol{\chi},\log p)}\ll & M_{2q}(n,\boldsymbol{\chi})^{(q-2)/(q-1)}C(n,\boldsymbol{\chi})\widehat{\tau_{2}}(n,\boldsymbol{\chi})^{1/(q-1)}\\
&+R_{2,q}(n,x){\rm ,}
\end{align*}
\noindent
où $C(n,\boldsymbol{\chi})$ est défini en \eqref{eq C},
$$R_{2,q}(n,x) \ll \e^{-c(\log x)^{\eta}}4^qM_{2q}(n)^{(2q-2)/(2q-1)}\tau_3(n)^{2q/(2q-1)} $$
\noindent
et
$$M_{h}(n):=\int_{\mathbb{R}^2}{\Delta_3(n,\mathbf{u})^h {\rm d}\mathbf{u}}{\rm ,} $$
avec $\Delta_3(n,\mathbf{u})$ défini en \eqref{eq Deltaru}. \\
De plus, si nous remplaçons $N_{1,j,q}(n,\boldsymbol{\chi},\log p)$ par  $N_{2,j,q}(n,\boldsymbol{\chi},\log p)$, alors la majoration reste valable en remplaçant $C(n,\boldsymbol{\chi})$ par  $D(n,\boldsymbol{\chi})$, défini en \eqref{eq D}.
\end{lemme}

\begin{proof}

Nous raisonnons uniquement sur la fonction $N_{1,j,q}$, la démonstration pour la fonction $N_{2,j,q}$ étant analogue. Pour tout $n\geqslant 1$, nous avons
\begin{align*}
 \sum\limits_{p>x}{\frac{g(p)\log p}{p}N_{1,j,q}(n,\boldsymbol{\chi},\log p)} =\int_{[0,1]^2}{\int_{\mathbb{R}^2}{\big|\Delta_3(n,\boldsymbol{\chi},\mathbf{u},\mathbf{v})\big|^{2j}S_{2(q-j)}{\rm d} \mathbf{u}}{\rm d} \mathbf{v}}{\rm ,}  
\end{align*}
\noindent
où 
$$S_{2h}:=\sum\limits_{p>x}{\frac{g(p)\log (p)}{p}|\Delta_3(n,\boldsymbol{\chi},u_1-\log(p),u_2,\mathbf{v})|^{2h}} {\rm .}$$

\noindent
Un développement de la somme définissant $\Delta_3$ fournit l'égalité suivante
\begin{align*}
S_{2h}=&\sum\limits_{d_1,\ldots, d_{2h} \mid n}{\chi_1(d_1\cdots d_{h})\overline{\chi_1(d_{h+1}\cdots d_{2h})}\prod\limits_{i=1}^{h}{\Delta\Big(\frac{n}{d_i},\chi_2,u_2,v_2\Big)\overline{\Delta\Big(\frac{n}{d_{h+i}},\chi_2,u_2,v_2\Big)}}}\\
&\sum\limits_{\substack{u_1-\log \min d_r<\log p\leqslant u_1+v_1-\log \max d_r \\ p>x}}{\frac{g(p)\log p}{p}} {\rm .}
\end{align*}
\noindent
Par sommation d'Abel, la somme intérieure vaut, dans le cas où~$\log( \max d_r/\min d_r)\leqslant~v_1$ 
$$y\int_{u_1-\log \min d_r}^{u_1+v_1-\log \max d_r}{1_{[0,1]}\Big(\frac{\log x}{t}\Big){\rm d}t}+O\big(\e^{-c(\log x)^{\eta}}\big)$$
où $y=y(g)$.
\noindent
De plus, elle est nulle si l'inégalité n'est pas respectée.
En adaptant le théorème 72 de \cite{HT}, nous obtenons, pour tout $h\geqslant 1$
$$\sum\limits_{\substack{d_1,\ldots, d_h \mid n \\ \log(\max d_i)-\log(\min d_i)\leqslant v_1}}{\prod\limits_{i=1}^{h}{\Delta\Big(\frac{n}{d_i},u_2\Big)}}\leqslant 2^hM_{h}(n,u_2){\rm ,}$$
\noindent
où 
$$M_h(n,u_2)=\int_{\mathbb{R}}{\Delta_3(n,\mathbf{u})^h{\rm d}u_1} {\rm .}$$
\noindent
Ainsi,
\begin{align*}
S_{2h} & =y\int_{\log x}^{+\infty}{|\Delta_3(n,\boldsymbol{\chi},u_1-t,u_2,\mathbf{v})|^{2h}{\rm d}t}+O\big(2^{2h}M_{2h}(n,u_2)\e^{-c(\log x)^{\eta}}\big) \\
& \leqslant y\int_{\mathbb{R}}{|\Delta_3(n,\boldsymbol{\chi},u_1-t,u_2,\mathbf{v})|^{2h}{\rm d}t}+O\big(2^{2h}M_{2h}(n,u_2)\e^{-c(\log x)^{\eta}}\big){\rm .}
\end{align*}
\noindent
Nous obtenons donc
\begin{equation}
\label{eq maj 5}
\begin{split}
&\sum\limits_{p>x}{\frac{g(p)\log p}{p}N_{1,j,q}(n,\log p)}\\
&\leqslant y Q_j+O\Big(2^{2(q-j)}\int_{\mathbb{R}}{M_{2(q-j)}(n,u_2)M_{2j}(n,u_2){\rm d}u_2}\e^{-c(\log x)^{\eta}}\Big){\rm ,}
\end{split}
\end{equation}
\noindent
où
$$Q_j:=\int_{[0,1]^2}{\int_{\mathbb{R}}\bigg({\int_{\mathbb{R}}{|\Delta_3(n,\boldsymbol{\chi},\mathbf{u'},\mathbf{v})|^{2j}{\rm d}u_1'}\int_{\mathbb{R}}{|\Delta_3(n,\boldsymbol{\chi},\mathbf{u},\mathbf{v})|^{2(q-j)}{\rm d}u_1}\bigg){\rm d}u_2}{\rm d}\mathbf{v}}{\rm ,}$$
avec $\mathbf{u'}=(u_1',u_2)$.
\goodbreak

\noindent
 Nous appliquons alors l'inégalité de Hölder avec les exposants $\frac{q-1}{q-j}$ et $\frac{q-1}{j-1}$ à
$$\int_{\mathbb{R}}{|\Delta_3(n,\boldsymbol{\chi},\mathbf{u},\mathbf{v})|^{2j}{\rm d}u_1}$$
\noindent
en écrivant $j=q\frac{j-1}{q-1}+\frac{q-j}{q-1}$ pour obtenir
\begin{align*}
& \int_{\mathbb{R}}{|\Delta_3(n,\boldsymbol{\chi},\mathbf{u},\mathbf{v})|^{2j}{\rm d}u_1}\\ 
& \leqslant \Big(\int_{\mathbb{R}}{|\Delta_3(n,\boldsymbol{\chi},\mathbf{u},\mathbf{v})|^{2q}{\rm d}u_1}\Big)^{(j-1)/(q-1)}\Big(\int_{\mathbb{R}}{|\Delta_3(n,\boldsymbol{\chi},\mathbf{u},\mathbf{v})|^{2}{\rm d}u_1}\Big)^{(q-j)/(q-1)} {\rm .}
\end{align*}
\noindent
En appliquant le même raisonnement en remplaçant $j$ par $q-j$, nous obtenons finalement
\begin{align*}
Q_j & \ll M_{2q}(n,\boldsymbol{\chi})^{(q-2)/(q-1)}\Big(\int_{[0,1]^2}{\int_{\mathbb{R}}{\Big(\int_{\mathbb{R}}{|\Delta_3(n,\boldsymbol{\chi},\mathbf{u},\mathbf{v})|^2{\rm d}u_1}\Big)^q{\rm d}u_2}{\rm d}\mathbf{v}}\Big)^{1/(q-1)} \\
 & \ll  M_{2q}(n,\boldsymbol{\chi})^{(q-2)/(q-1)}C(n,\boldsymbol{\chi})M_2(n,\boldsymbol{\chi})^{1/(q-1)}{\rm .}
\end{align*}
\noindent
Une adaptation du lemme $2.3$ de  \cite{B} fournit
$$M_2(n,\boldsymbol{\chi})\ll \widehat{\tau_{2}}(n,\boldsymbol{\chi}) {\rm .}$$
\noindent
Pour traiter la contribution du terme d'erreur du majorant de \eqref{eq maj 5}, nous effectuons un raisonnement similaire. Le terme~$|\Delta_3(n,\boldsymbol{\chi},\mathbf{u},\mathbf{v})|$ est clairement majoré par $\Delta_3(n,\mathbf{u})$. Puis, une inégalité de Hölder avec exposants $\frac{2q-1}{2(q-j)}$ et $\frac{2q-1}{2j-1}$ montre l'inégalité 
\begin{align*}
&\int_{\mathbb{R}}{\Delta_3(n,\mathbf{u})^{2j}{\rm d}u_1} \\
\leqslant& \Big(\int_{\mathbb{R}}{\Delta_3(n,\mathbf{u})^{2}{\rm d}u_1}\Big)^{2(q-j)/(2q-1)}\Big(\int_{\mathbb{R}}{\Delta_3(n,\mathbf{u})^{2q}{\rm d}u_1}\Big)^{(2h-1)/(2q-1)} {\rm .}
\end{align*}
\noindent
Le même raisonnement en remplaçant $j$ par $q-j$ fournit
\begin{align*}
& \int_{\mathbb{R}}{\int_{\mathbb{R}}{\Delta_3(n,\mathbf{u})^{2j}{\rm d}u_1}\int_{\mathbb{R}}{\Delta_3(n,\mathbf{u})^{2(q-j)}{\rm d}u_1}{\rm d}u_2} \\ 
\leqslant &\int_{\mathbb{R}}{\Big(\int_{\mathbb{R}}{\Delta_3(n,\mathbf{u})^{2q}{\rm d}u_1}\Big)^{(2q-2)/(2q-1)}\Big(\int_{\mathbb{R}}{\Delta_3(n,\mathbf{u}){\rm d}u_1}\Big)^{2q/(2q-1)}{\rm d}u_2} \\
\leqslant & M_{2q}(n)^{(2q-2)/(2q-1)}\Big(\int_{\mathbb{R}}{\Big(\int_{\mathbb{R}}{\Delta_3(n,\mathbf{u}){\rm d}u_1}\Big)^{2q}{\rm d}u_2}\Big)^{1/(2q-1)} {\rm .}
\end{align*}
\noindent
Pour conclure, nous utilisons l'estimation 
\begin{align*}
\int_{\mathbb{R}}{\Big(\int_{\mathbb{R}}{\Delta_3(n,\mathbf{u}){\rm d}u_1}\Big)^{2q}{\rm d}u_2}&= \sum\limits_{d_1,\ldots,d_{2q}\mid n}{\prod\limits_{i=1}^{2q}{\tau\Big(\frac{n}{d_i}\Big)}\max\bigg(0,1-\log\Big(\frac{\max(d_i)}{\min(d_i)}\Big)\bigg)} \\
& \leqslant \sum\limits_{d_1,\ldots,d_{2q}\mid n}{\prod\limits_{i=1}^{2q}{\tau\Big(\frac{n}{d_i}\Big)}} \\
& \leqslant \tau_3(n)^{2q} {\rm .}
\end{align*}
\end{proof}
Pour le prochain lemme, nous notons, lorsque $\chi$ est un caractère de Dirichlet non principal, 
\begin{equation}
\label{Ndag}
N^{\dagger}_{j,q}(n,\chi,w):=\sum\limits_{d\mid n}{\int_{0}^{1}{\int_{\mathbb{R}}{|\Delta(d,\chi,u,v)|^{2j}|\Delta(d,\chi,u-w,v)|^{2(q-j)}{\rm d}u}{\rm d}v}}{\rm ,}
\end{equation}
et
\begin{equation}
\label{eq tau0}
\widehat{\tau_{0}}(n,\chi):=\Big(\sum\limits_{d\mid n}{\Big(\int_{\mathbb{R}}{|\tau(d,\chi,\vartheta)|^2\frac{{\rm d}\vartheta}{1+\vartheta^2}}\Big)^q}\Big)^{1/q}{\rm ,}
\end{equation}
où $\tau(n,\chi,\vartheta)$ est défini en \eqref{eq tautheta}. Enfin, nous définissons
$$M^{\dagger}_{h}(n):=\sum\limits_{d\mid n}{M_h(d)}{\rm ,}$$
où
$$M_h(d):=\int_{\mathbb{R}}{\Delta(d,u)^h {\rm d}u}{\rm .}$$
\begin{lemme}
\label{lemme 6}
Soient $A$, $c$ des constantes strictement positives,  $\eta \in ]0,1[$ et $\chi$ un caractère de Dirichlet non principal. Pour toute fonction $g$ de $\mathcal{M}_A(c,\eta)$, $q\geqslant 1$, $1\leqslant j \leqslant q-1$, $n\geqslant 1$ et $x\geqslant 2$, nous avons 
\begin{align*}
\sum\limits_{p>x}{\frac{g(p)\log p}{p}N^{\dagger}_{j,q}(n,\chi,\log p)}\ll & M^{\dagger}_{2q}(n,\chi)^{q-2/q-1}\widehat{\tau_{0}}(n,\chi)^{q/q-1}\\
&+R^{\dagger}_{q}(n,x){\rm ,}
\end{align*}
\noindent
où
$$R^{\dagger}_{q}(n,x) \ll \e^{-c(\log x)^{\eta}}4^qM^{\dagger}_{2q}(n)^{(2q-2)/(2q-1)}\tau_3(n)^{2q/(2q-1)}{\rm ,}$$
$N^{\dagger}_{j,q}(n,\chi,\log p)$ est défini en \eqref{Ndag} et $\widehat{\tau_0}(n,\chi)$ est défini en \eqref{eq tau0}.
\end{lemme}

\begin{proof}
Nous reprenons la démonstration du Lemme \ref{lemme 5}. Pour tout~$n\geqslant~1$, nous avons
\begin{align*}
 \sum\limits_{p>x}{\frac{g(p)\log p}{p}N^{\dagger}_{j,q}(n,\chi,\log p)} =\sum\limits_{d \mid n}{\int_{0}^{1}{\int_{\mathbb{R}}{\big|\Delta(d,\chi,u,v)\big|^{2j}S^*_{2(q-j)}(d){\rm d} u}{\rm d} v} } {\rm ,}
\end{align*}
\noindent
où 
$$S^*_{2h}(d):=\sum\limits_{p>x}{\frac{g(p)\log (p)}{p}|\Delta(d,\chi,u-\log(p),v)|^{2h}} {\rm .}$$
\noindent
Le développement de $S^*_{2h}(d)$ et l'application du théorème des nombres premiers fournissent
$$S^*_{2h}(d) \leqslant y\int_{\mathbb{R}}{|\Delta(d,\chi,u-t,v)|^{2h}{\rm d}t}+O\Big(2^{2h}M_{2h}(d)\e^{-c(\log x)^{\eta}}\Big)$$
où $y=y(g)$. Nous obtenons donc
$$\sum\limits_{p>x}{\frac{g(p)\log p}{p}N^{\dagger}_{j,q}(n,\chi,\log p)}\leqslant  yQ^{\dagger}_j+O\Big(2^{2(q-j)}\sum\limits_{d\mid n}{M_{2j}(d)M_{2(q-j)}(d)}\e^{-c(\log x)^{\eta}}\Big){\rm ,}$$
où
$$Q^{\dagger}_j:=\sum\limits_{d\mid n}{\int_0^1{\int_{\mathbb{R}}{|\Delta(d,\chi,u,v)|^{2j}{\rm d}u}\int_{\mathbb{R}}{|\Delta(d,\chi,u,v)|^{2(q-j)}{\rm d}u}{\rm d}v}} {\rm .}$$
\goodbreak

\noindent
Des inégalités de Hölder impliquent
\begin{align*}
Q^{\dagger}_j & \ll \sum\limits_{d \mid n}{M_{2q}(d,\chi)^{(q-2)/(q-1)}\widehat{\tau}(d,\chi)^{q/(q-1)}}\\
    & \ll M^{\dagger}_{2q}(n,\chi)^{(q-2)/(q-1)}\widehat{\tau_{0}}(n,\chi)^{q/(q-1)}{\rm ,}
\end{align*}
\noindent
où
$$\widehat{\tau}(n,\chi)=\int_{\mathbb{R}}{|\tau(n,\chi,\vartheta)|^2\frac{{\rm d}\vartheta}{1+\vartheta^2}}{\rm .}$$
\noindent
Le terme d'erreur se traite de manière analogue au Lemme \ref{lemme 5}.
\end{proof}
Pour le prochain lemme, nous définissons, pour $k \in \mathbb{N}$, $n \geqslant 1$, $\boldsymbol{\chi}$ un couple de caractères de Dirichlet, $1\leqslant j<q$ et $(w,\vartheta) \in \mathbb{R}^2$ 
\begin{align}
\begin{split}
\label{Nijkq}
&N^{(k)}_{j,q}(n,\boldsymbol{\chi},w,\vartheta)\\
:=&\int_{[0,1]^2}{\int_{\mathbb{R}}{|\Delta^{(k)}(n,\boldsymbol{\chi},\vartheta,u,v_1,v_2)|^{2j}|\Delta^{(k)}(n,\boldsymbol{\chi},\vartheta,u-w,v_1,v_2)|^{2(q-j)}{\rm d}u}{\rm d}v_1{\rm d}v_2}{\rm ,}
\end{split}
\end{align}
où $\Delta^{(k)}(n,\boldsymbol{\chi},\vartheta,u,v_1,v_2)$ est défini en \eqref{Dktheta}. Enfin, nous notons
\begin{equation}
\label{eq M1}
M^{(1)}_{h}(n,\tau):=\int_{\mathbb{R}}{\Delta(n,\tau,u,1)^h {\rm d}u}
\end{equation}
où $\Delta(n,\tau,u,1)$ est défini en \eqref{eq Deltafuv} et $\tau$ est la fonction nombre de diviseurs.\\
\begin{lemme}
\label{lemme 7}
Soient $A$, $c$ des constantes strictement positives, $\eta \in ]0,1[$, $k\in \mathbb{N}$ et $\boldsymbol{\chi}$ un couple de caractères de Dirichlet. Pour toute fonction $g$ de $\mathcal{M}_A(c,\eta)$, $q\geqslant 1$, $1\leqslant j \leqslant q-1$, $n\geqslant 1$, $\vartheta \in \mathbb{R}$ et $x\geqslant 2$, nous avons
\begin{align*}
\sum\limits_{p>x}{\frac{g(p)\log p}{p}N^{(k)}_{j,q}(n,\boldsymbol{\chi},\vartheta,\log p)}\ll & (k+1)^{2q/(q-1)}M^{(k)}_{2q}(n,\boldsymbol{\chi},\vartheta)^{q-2/q-1}\widehat{\tau_{1}}(n,\boldsymbol{\chi},\vartheta)^{q/q-1}\\
&+R_{k,q}(n,x){\rm ,}
\end{align*}
\noindent
où $M_{2q}^{(k)}(n,\boldsymbol{\chi},\vartheta)$ est défini en \eqref{eq Mijtheta}, $\widehat{\tau_{1}}(n,\boldsymbol{\chi},\vartheta)$ est défini en \eqref{hattau1}, et
$$R_{k,q}(n,x) \ll (k+1)q\e^{-c(\log x)^{\eta}}4^qM^{(1)}_{2q}(n,\tau)^{(2q-2)/(2q-1)}\tau_3(n)^{2q/(2q-1)}{\rm .} $$

\end{lemme}

\begin{proof}
Pour tout $n\geqslant 1$, nous avons
\begin{align*}
 \sum\limits_{p>x}{\frac{g(p)\log p}{p}N^{(k)}_{j,q}(n,\boldsymbol{\chi},\vartheta,\log p)} =\int_{[0,1]^2}{\int_{\mathbb{R}}{\big|\Delta^{(k)}(n,\boldsymbol{\chi},\vartheta,u,v_1,v_2)\big|^{2j}S^{(k)}_{2(q-j)}{\rm d} u}{\rm d}v_1{\rm d}v_2}{\rm ,}  
\end{align*}
\noindent
où 
$$S^{(k)}_{2h}:=\sum\limits_{p>x}{\frac{g(p)\log (p)}{p}|\Delta^{(k)}(n,\boldsymbol{\chi},\vartheta,u-\log(p),v_1,v_2)|^{2h}} {\rm .}$$
\noindent
Un développement de la somme définissant $\Delta^{(k)}$ fournit l'égalité suivante
\begin{align*}
S^{(k)}_{2h}=&\sum\limits_{d_1,\ldots, d_{2h} \mid n}{\chi_1(d_1\cdots d_{h})\overline{\chi_1(d_{h+1}\cdots d_{2h})}\prod\limits_{i=1}^{h}{\tau\Big(\frac{n}{d_i},\chi_2,\vartheta\Big)\overline{\tau\Big(\frac{n}{d_{h+i}},\chi_2,\vartheta\Big)}}}\\
&\sum\limits_{\substack{u-\log \min d_r<\log p\leqslant u+v_1-\log \max d_r \\ p>x}}{\frac{g(p)\log p}{p}\prod\limits_{i=1}^{2h}{(u+v_1-v_2-\log p-\log d_i)^k}} {\rm .}
\end{align*}
\noindent
Par sommation d'Abel, la somme intérieure vaut, dans le cas où~$\log( \max d_r/\min d_r)\leqslant~v_1$ 
$$y\int_{u_1-\log \min d_r}^{u_1+v_1-\log \max d_r}{1_{[0,1]}\Big(\frac{\log x}{t}\Big)\prod\limits_{i=1}^{2h}{(u+v_1-v_2-t-\log d_i)^k}{\rm d}t}+O\big((k+1)h\e^{-c(\log x)^{\eta}}\big)$$
où $y=y(g)$.
Le point important étant que tous les termes 
$$(u+v_1-v_2-t-\log d_i)^k$$
sont majorés par $1$ en valeur absolue dans les intervalles, en $t$, considérés, ce qui permet d'obtenir le terme d'erreur ci-dessus. \\
\noindent
De plus, elle est nulle si l'inégalité n'est pas respectée.
En adaptant le théorème 72 de \cite{HT}, nous obtenons, pour tout $h\geqslant 1$
$$\sum\limits_{\substack{d_1,\ldots, d_h \mid n \\ \log(\max d_i)-\log(\min d_i)\leqslant v_1}}{\prod\limits_{i=1}^{h}{\tau\Big(\frac{n}{d_i}\Big)}}\leqslant 2^hM^{(1)}_{h}(n,\tau){\rm ,}$$
où $M^{(1)}_{h}(n,\tau)$ est défini en \eqref{eq M1}.
Ainsi,
\begin{align*}
S^{(k)}_{2h} & =y\int_{\log x}^{+\infty}{|\Delta^{(k)}(n,\boldsymbol{\chi},\vartheta,u-t,v_1,v_2)|^{2h}{\rm d}t}+O\big((k+1)h2^{2h}M^{(1)}_{2h}(n,\tau)\e^{-c(\log x)^{\eta}}\big) \\
& \leqslant y\int_{\mathbb{R}}{|\Delta^{(k)}(n,\boldsymbol{\chi},\vartheta,u-t,v_1,v_2)|^{2h}{\rm d}t}+O\big((k+1)h2^{2h}M^{(1)}_{2h}(n,\tau)\e^{-c(\log x)^{\eta}}\big){\rm .}
\end{align*}
\noindent
Nous obtenons donc
\begin{equation}
\label{eq maj 5bis}
\begin{split}
&\sum\limits_{p>x}{\frac{g(p)\log p}{p}N^{(k)}_{j,q}(n,\vartheta,\log p)}\\
&\leqslant y Q_j+O\Big((k+1)(q-j)2^{2(q-j)}M^{(1)}_{2(q-j)}(n,\tau)M^{(1)}_{2j}(n,\tau)\e^{-c(\log x)^{\eta}}\Big){\rm ,}
\end{split}
\end{equation}
\noindent
où
$$Q_j:=\int_{[0,1]^2}{\bigg(\int_{\mathbb{R}}{|\Delta^{(k)}(n,\boldsymbol{\chi},\vartheta,u',v_1,v_2)|^{2j}{\rm d}u'}\int_{\mathbb{R}}{|\Delta^{(k)}(n,\boldsymbol{\chi},\vartheta,u,v_1,v_2)|^{2(q-j)}{\rm d}u}\bigg){\rm d}v_1{\rm d}v_2}{\rm .}$$
\pagebreak

\noindent
 Nous appliquons alors l'inégalité de Hölder avec les exposants $\frac{q-1}{q-j}$ et $\frac{q-1}{j-1}$ à
$$\big \lvert\Delta^{(k)}(n,\boldsymbol{\chi},\vartheta,u,v_1,v_2)\big \rvert^{2j}$$
\noindent
en écrivant $j=q\frac{j-1}{q-1}+\frac{q-j}{q-1}$ pour obtenir
\begin{align*}
& \int_{\mathbb{R}}{|\Delta^{(k)}(n,\boldsymbol{\chi},\vartheta,u,v_1,v_2)|^{2j}{\rm d}u}\\ 
& \leqslant \Big(\int_{\mathbb{R}}{|\Delta^{(k)}(n,\boldsymbol{\chi},\vartheta,u,v_1,v_2)|^{2q}{\rm d}u}\Big)^{(j-1)/(q-1)}\Big(\int_{\mathbb{R}}{|\Delta^{(k)}(n,\boldsymbol{\chi},\vartheta,u,v_1,v_2)|^{2}{\rm d}u}\Big)^{(q-j)/(q-1)} {\rm .}
\end{align*}
\noindent
En appliquant le même raisonnement en remplaçant $j$ par $q-j$, nous obtenons finalement
\begin{align*}
Q_j & \ll M^{(k)}_{2q}(n,\boldsymbol{\chi},\vartheta)^{(q-2)/(q-1)}\Big(\int_{[0,1]^2}{\Big(\int_{\mathbb{R}}{|\Delta^{(k)}(n,\boldsymbol{\chi},\vartheta,u,v_1,v_2)|^2{\rm d}u}\Big)^q{\rm d}v_1{\rm d}v_2}\Big)^{1/(q-1)} {\rm .}
\end{align*}
\noindent
Le Lemme \ref{lemme tau1}  fournit la majoration suivante, uniformément pour $v_1,v_2 \in [0,1]^2$.
$$\int_{\mathbb{R}}{|\Delta^{(k)}(n,\boldsymbol{\chi},\vartheta,u,v_1,v_2)|^2{\rm d}u}\ll (k+1)^2 \widehat{\tau_{1}}(n,\boldsymbol{\chi},\vartheta) {\rm .}$$
\noindent
Pour traiter la contribution du terme d'erreur du majorant de \eqref{eq maj 5bis}, nous appliquons une inégalité de Hölder avec exposants $\frac{2q-1}{2(q-j)}$ et $\frac{2q-1}{2j-1}$ afin d'obtenir
\begin{align*}
&\int_{\mathbb{R}}{\Delta(n,\tau,u)^{2j}{\rm d}u} \\
\leqslant& \Big(\int_{\mathbb{R}}{\Delta(n,\tau,u)^{2}{\rm d}u}\Big)^{2(q-j)/(2q-1)}\Big(\int_{\mathbb{R}}{\Delta(n,\tau,u)^{2q}{\rm d}u}\Big)^{(2h-1)/(2q-1)} {\rm .}
\end{align*}
\noindent
Le même raisonnement en remplaçant $j$ par $q-j$ fournit
\begin{align*}
& M^{(1)}_{2j}(n,\tau)M^{(1)}_{2(q-j)}(n,\tau) \\ 
\leqslant &\Big(\int_{\mathbb{R}}{\Delta(n,\tau,u)^{2q}{\rm d}u}\Big)^{(2q-2)/(2q-1)}\Big(\int_{\mathbb{R}}{\Delta(n,\tau,u){\rm d}u}\Big)^{2q/(2q-1)}\\
\leqslant & M^{(1)}_{2q}(n,\tau)^{(2q-2)/(2q-1)}\tau_3(n)^{2q/(2q-1)} {\rm .}
\end{align*}
\end{proof}

\subsection{Estimations de sommes sur les nombres premiers}
Les lemmes présentés dans cette sous-section sont des lemmes arithmétiques fondamentaux pour les différentes démonstrations de cet article.
\begin{lemme}
\label{lemme 8}
Soient $\chi$ un caractère de Dirichlet non principal d'ordre $r$, $A$, $c$ des constantes strictement positives et $\eta \in ]0,1[$ . Pour $T>1$ et $g$ appartenant $\mathcal{M}_A(\chi,c,\eta)$, nous avons, uniformément pour $\sigma>0$, $1\leqslant q\leqslant \frac{1}{\sigma}$ et $|\vartheta| \leqslant q\sigma$
\begin{equation}
\label{eq 8.1}
\sum\limits_{p>T}{\max\big(1,|1+\chi(p)p^{i\vartheta}|^2\big)g(p)\frac{\log p}{p^{1+\sigma}}}\leqslant \frac{3y}{\sigma}+O(1){\rm ,}
\end{equation}
 où $y=y(g)$. Nous avons uniformément pour $q\sigma \leqslant |\vartheta| \leqslant \e^{c(\log T)^{\eta}}$
\begin{equation}
\label{eq 8.2}
\sum\limits_{p>T}{\max\big(1,|1+\chi(p)p^{i\vartheta}|^2\big)g(p)\frac{\log p}{p^{1+\sigma}}}=\frac{(\rho+2)y+O(1/q)}{\sigma}{\rm ,}
\end{equation}
où $\rho$ est défini en \eqref{eq lambda}, et uniformément pour $|\vartheta| \geqslant \e^{c(\log T)^{\eta}}$
\begin{equation}
\label{eq 8.3}
\sum\limits_{p>T}{|1+\chi(p)p^{i\vartheta}|^2g(p)\frac{\log p}{p^{1+\sigma}}}\leqslant \frac{4y}{\sigma}+O(1){\rm .}
\end{equation}
\end{lemme}

\begin{proof}
 
Lorsque $|\vartheta|\leqslant q\sigma$, nous utilisons l'inégalité suivante
$$\max\big(1,|1+\chi(p)p^{i\vartheta}|^2\big)\leqslant 1+|1+\chi(p)p^{i\vartheta}|^2{\rm ,}$$
ce qui fournit
\begin{align*}
\sum\limits_{p>T}{\max\big(1,|1+\chi(p)p^{i\vartheta}|^2\big)g(p)\frac{\log p}{p^{1+\sigma}}}& \leqslant \sum\limits_{p>T}{\big(3+\chi(p)p^{i\vartheta}+\overline{\chi(p)}p^{-i\vartheta}\big)g(p)\frac{\log p}{p^{1+\sigma}}}{\rm .}
\end{align*}
\noindent
D'après \eqref{eq def}, nous avons d'une part
$$\sum\limits_{p>T}{3g(p)\frac{\log p}{p^{1+\sigma}}}\leqslant \frac{3y}{\sigma}+O(1)$$
\noindent
et d'autre part, d'après \eqref{eq g}
\begin{align*}
\sum\limits_{p>T}{\chi(p)p^{i\vartheta}g(p)\frac{\log p}{p^{1+\sigma}}} &=\sum\limits_{k=0}^{r-1}{\sum\limits_{\substack{p>T \\ \chi(p)=\zeta^k}}{\zeta^kp^{i\vartheta}g(p)\frac{\log p}{p^{1+\sigma}}}}\\
&=\frac{y}{r}\sum\limits_{k=0}^{r-1}{\zeta^{k}\Big(\int_{\log T}^{+\infty}{\e^{(-\sigma+i\vartheta)u}{\rm d}u}+O(|\vartheta|\e^{-c(\log T)^{\eta}})\Big)}\\
&=\frac{y}{r}\sum\limits_{k=0}^{r-1}{\zeta^{k}\frac{T^{-\sigma+i\vartheta}}{\sigma-i\vartheta}}+O\big(|\vartheta|\e^{-c(\log T)^{\eta}}\big) \\
&=O\big(|\vartheta|\e^{-c(\log T)^{\eta}}\big) 
\end{align*}
ce qui fournit \eqref{eq 8.1} puisque $|\vartheta| \leqslant \e^{c(\log T)^{\eta}}$. \\
\indent
Si $q\sigma \leqslant |\vartheta| \leqslant \e^{c(\log T)^{\eta}}$, nous observons que la fonction
$$t \mapsto \max\big(1,|1+\e^{it}|^2\big)$$
est périodique, de période $2\pi$, et de moyenne $\rho+2$. Une intégration par partie semblable au lemme \cRM{3}. 4.13 de \cite{T} fournit
$$\sum\limits_{p>T}{\max\big(1,|1+\chi(p)p^{i\vartheta}|^2\big)g(p)\frac{\log p}{p^{1+\sigma}}}= \frac{(\rho+2)y}{\sigma}+O\Big(\frac{1}{|\vartheta|}+\frac{|\vartheta|+1}{\e^{c(\log T)^{\eta}}}\Big){\rm ,}$$
ce qui implique \eqref{eq 8.2}.
Dans le dernier cas, nous nous contentons de la majoration triviale
$$|1+\chi(p)p^{i\vartheta}|^2\leqslant 4 {\rm .}$$
\end{proof}

Pour $r\geqslant 1$, nous notons
\begin{equation}
\label{eq kappa}
\kappa(r):=\frac{1}{r}\sum\limits_{k=0}^{r-1}{\max\big(1,|1+\zeta^k|^2\big)}
\end{equation}
où $\zeta$ est défini en \eqref{eq zeta}.
\begin{lemme}
\label{lemme 9}
Soient $\chi$ un caractère de Dirichlet non principal d'ordre $r$, $A$, $c$ des constantes strictement positives, $\eta \in ]0,1[$ et $g$ appartenant $\mathcal{M}_A(\chi,c,\eta)$ . Nous avons, uniformément pour $x\geqslant 16$ et $0<|\vartheta|\leqslant 1$  ,
\begin{align*}
\sum\limits_{p\leqslant x}{g(p)\frac{\max\big(1,|1+\chi(p)p^{i\vartheta}|^2\big)}{p}}=& \;\;\;\;y\kappa(r)\log\Big(\frac{\log x}{1+|\vartheta|\log x}\Big) \\
&+y(\rho+2) \log\big(1+|\vartheta|\log x\big)+O(1){\rm ,}
\end{align*}
\noindent 
où $y=y(g)$, et uniformément pour $|\vartheta|>1$
$$\sum\limits_{p\leqslant x}{g(p)\frac{\max\big(1,|1+\chi(p)p^{i\vartheta}|^2\big)}{p}}\leqslant y(\rho+2) \log\big(1+|\vartheta|\log x\big)+B\log_2(2+|\vartheta|){\rm ,}$$
\noindent
où $\rho$ est défini en \eqref{eq lambda}.
\end{lemme}

\begin{proof}
La preuve est analogue à celle du lemme $2.5$ de \cite{B}. La fonction 
$$t\mapsto \max\big(1,|1+\e^{it}|^2\big)$$ 
est périodique, de moyenne $\rho+2$. Ainsi, si nous nous donnons un paramètre $w\leqslant x$, une intégration par partie semblable à celle du lemme \cRM{3}. 4.13 de \cite{T} fournit
$$\sum\limits_{w<p\leqslant x}{\frac{g(p)\max\big(1,|1+\chi(p)p^{i\vartheta}|^2\big)}{p}}=y(\rho+2) \log\Big(\frac{\log x}{\log w}\Big)+O\Big(\frac{1}{|\vartheta| \log w}+\frac{1+|\vartheta|}{\e^{c(\log w)^{\eta}}}\Big){\rm .}$$
\noindent
Pour les entiers $p\leqslant w$, nous nous contentons d'approcher $p^{i\vartheta}$ par $1$, l'erreur que nous commettons est alors $O(|\vartheta| \log w)$. Ainsi
$$\sum\limits_{p\leqslant w}{g(p)\frac{\max\big(1,|1+\chi(p)p^{i\vartheta}|^2\big)}{p}}=y\kappa(r)\log_2(w)+O(1+|\vartheta| \log w){\rm .}$$
Il s'agit désormais de choisir le paramètre $w$ afin d'obtenir les formules souhaitées. Pour $|\vartheta| \leqslant \frac{1}{\log x}$, nous choisissons $w=x$.\\ \indent
Pour $1/\log x \leqslant |\vartheta| \leqslant 1$, nous choisissons $w=\exp\big(\frac{\log x}{1+|\vartheta| \log x}\big)$. \\
\indent
Pour $1 \leqslant |\vartheta| \leqslant \exp\big(c(\log x)^{\eta}\big)-2$, nous choisissons $w=\exp\big((2/c)(\log (2+~|\vartheta|))^{1/\eta}\big)$. \\
\indent
Dans le dernier cas, nous nous contentons d'observer que la somme à majorer est $\ll \log_2(x)\ll \log_2(2+|\vartheta|)$.
\end{proof}

\begin{lemme}
\label{lemme 10}
Soient $\chi$ un caractère de Dirichlet d'ordre $r>1$ $A$, $c$ des constantes strictement positives, $\eta \in ]0,1[$ et $g$ appartenant à $\mathcal{M}_A(\chi,c,\eta)$. Nous avons uniformément pour  $x\geqslant 16$ et $|\vartheta|\leqslant 1$
$$\sum\limits_{p\leqslant x}{g(p)\frac{\chi(p)}{p}p^{i\vartheta}} =O(1)$$
\noindent
et uniformément pour $|\vartheta| >1$
$$\sum\limits_{p\leqslant x}{g(p)\frac{\chi(p)}{p}p^{i\vartheta}} \leqslant B\log_2(2+|\vartheta|){\rm .}$$
\end{lemme}

\begin{proof}
Notant $\zeta$ le complexe défini en \eqref{eq zeta}, nous avons
$$\sum\limits_{p\leqslant x}{g(p)\frac{\chi(p)}{p}p^{i\vartheta}}=\sum\limits_{k=0}^{r-1}{\sum\limits_{\substack{p\leqslant x \\ \chi(p)=\zeta^k}}{g(p)\frac{\zeta^k}{p}p^{i\vartheta}}}{\rm .}$$
La suite de la démonstration est identique à celle du lemme $2.5$ de \cite{B}, nous introduisons un paramètre $w$, choisi convenablement par la suite. La fonction~$t\mapsto\e^{it}$ est périodique de période nulle, donc le lemme \cRM{3}.4.13 de \cite{T} permet d'écrire
$$\sum\limits_{\substack{w<p\leqslant x \\ \chi(p)=\zeta^k}}{\frac{g(p)}{p}p^{i\vartheta}}=O\Big(\frac{1}{|\vartheta| \log w}+\frac{1+|\vartheta|}{\e^{c(\log w)^{\eta}}}\Big){\rm .}$$
\noindent
Par ailleurs, pour $p\leqslant w$, nous approchons $p^{i\vartheta}$ par $1$, l'erreur que nous commettons étant $O(|\vartheta| \log w)$. Nous renvoyons à la démonstration du Lemme \ref{lemme 9} pour savoir comment choisir $w$ afin que les différents termes d'erreur soient~$O(1)$ lorsque $|\vartheta|\leqslant 1$ et $O\big(\log_2(2+|\vartheta|)\big)$ lorsque $|\vartheta|>1$. Nous concluons en remarquant que 
$\sum\limits_{k=0}^{r-1}{\zeta^k}=0$. 
\end{proof}

\begin{lemme}
\label{lemme 11}
Soit $\boldsymbol{\chi}$ un couple de deux caractères de Dirichlet non principaux tels que $\chi_1\overline{\chi_2}$ ne soit pas principal. Soient $A$, $c$ des constantes strictement positives, $\eta \in ]0,1[$ et $g$ appartenant à $\mathcal{M}_A(\boldsymbol{\chi},c,\eta)$. Pour $x\geqslant 16$ et uniformément pour $(\vartheta_1,\vartheta_2) \in \mathbb{R}^2$, nous avons
\begin{align*}
\sum\limits_{p\leqslant x} &{g(p)\frac{|1+\chi_1(p)p^{i\vartheta_1}+\chi_2(p)p^{i\vartheta_2}|^2}{p}}\\ &\leqslant 3y\log_2(x)\\
&\;\;\;+B\big(\log_2(2+|\vartheta_1|)1_{]1,+\infty[}(|\vartheta_1|)+\log_2(2+|\vartheta_2|)1_{]1,+\infty[}(|\vartheta_2|)\big) +O(1) {\rm ,}
\end{align*}
où $y=y(g)$.
\end{lemme}

\begin{proof}
Un développement du carré du module fournit
\begin{align*}
|1+\chi_1(p)p^{i\vartheta_1}+\chi_2(p)p^{i\vartheta_2}|^2= & 3+\chi_1(p)p^{i\vartheta_1}+\overline{\chi_1(p)}p^{-i\vartheta_1} \\
& +\chi_2(p)p^{i\vartheta_2}+\overline{\chi_2(p)}p^{-i\vartheta_2}\\
&+\chi_1(p)\overline{\chi_2(p)}p^{i(\vartheta_1-\vartheta_2)}+\chi_2(p)\overline{\chi_1(p)}p^{i(\vartheta_2-\vartheta_1)} {\rm .}
\end{align*}
\noindent
Nous étudions alors les sommes sur $p$ de chacun de ces termes séparément, en utilisant le Lemme \ref{lemme 10} pour obtenir le résultat énoncé.
\end{proof}

\subsection{Lemmes fondamentaux}
Le terme $\kappa(r)$ défini en \eqref{eq kappa} peut facilement être majoré par $3$, ce majorant est suffisant pour la suite, néanmoins, nous pouvons l'améliorer, au moyen de la proposition suivante.
\begin{lemme}
\label{prop}
Pour tout entier $r \geqslant 2$, nous avons 
$$\kappa(r)\leqslant \frac{5}{2} {\rm ,}$$
\noindent
où $\kappa(r)$ est défini en \eqref{eq kappa}. Cette inégalité est une égalité lorsque $r=2$.
\end{lemme}

\begin{proof}
Nous notons $A_r=\{z \in \mathbb{U}_r \;:\; |1+z|<1\}$ où $\mathbb{U}_r$ désigne l'ensemble des racines r-iémes de l'unité. Ainsi
\begin{align*}
\kappa(r)&=\frac{1}{r}\bigg(\sum\limits_{z \notin A}{|1+z|^2}+\sum\limits_{z \in A}{1}\bigg)\\
&=\frac{1}{r}\bigg(\sum\limits_{z \in \mathbb{U}_r}{|1+z|^2}+\sum\limits_{z \in A}{(1-|1+z|^2)}\bigg)\\
&=2+\frac{1}{r}\sum\limits_{z \in A}{\big(1-|1+z|^2\big)}{\rm .}
\end{align*}
Si $A_r$ est non vide, nous pouvons choisir $z \in A_r$, alors $\rm{arg}(z) \in ]2\pi/3,4\pi/3[$, ainsi, pour tout $z' \in A_r$, $\rm{arg}(zz')\in ]4\pi/3,8\pi/3[$, donc $zz' \notin A_r$, nous avons donc une injection de $A_r$ dans $\mathbb{U}_r\smallsetminus A_r$, ce qui montre que $|A_r|\leqslant \frac{r}{2}$. Ainsi
$$\kappa(r)\leqslant \frac{5}{2}{\rm .}$$
\end{proof}
Par la suite, pour $T>1$, pour tout $n\geqslant 1$, nous notons
\begin{equation}
\label{eq anbn}
a_n:=\prod\limits_{\substack{p\mid n \\ p\leqslant T}}{p} \;\;\;\;\;\;\;\;\;\;\;\;\;\;\;\;\;\;\;\;\;\;\;\;  b_n:=\prod\limits_{\substack{p\mid n \\ p> T}}{p} {\rm .}
\end{equation}

\begin{lemme}
\label{lemme 12}
Soient $\chi$ un caractère de Dirichlet non principal , $T>1$ un réel, $A$, $c$ des constantes strictement positives, $\eta \in ]0,1[$ et $g \in \mathcal{M}_A(\chi, c,\eta)$. Nous avons, uniformément pour $0<\sigma\leqslant \frac{1}{\log T}$ 
$$\int_{\mathbb{R}}{\sum\limits_{n\geqslant 1}{g(n)\frac{\mu^2(n)}{a_nb_n^{1+\sigma}}\max\limits_{d\mid n}\big(|\tau(d,\chi,\vartheta)|^2\big)}\frac{{\rm d}\vartheta}{1+\vartheta^2}} \ll\frac{1}{\sigma^{m(y,\rho)}} {\rm ,}$$
\noindent 
où $\rho$ est défini en \eqref{eq lambda} et $m(y,\rho)$ en \eqref{eq m}.
\end{lemme}

\begin{proof}
Nous ne traitons que l'intégrale sur $\mathbb{R}^+$ car celle sur $\mathbb{R}^-$ se déduit par symétrie. Notons $S(\vartheta)$ la somme à l'intérieur de l'intégrale. Elle admet un développement en produit eulérien
$$S(\vartheta)=\prod\limits_{p\leqslant T}{\Big(1+g(p)\frac{\max(1,|1+\chi(p)p^{i\vartheta}|^2)}{p}\Big)} \prod\limits_{p> T}{\Big(1+g(p)\frac{\max(1,|1+\chi(p)p^{i\vartheta}|^2)}{p^{1+\sigma}}\Big)} {\rm .}$$
\noindent
Le Lemme \ref{lemme 9} nous permet alors de majorer uniformément $S(\vartheta)$ en fonction de~$\vartheta$. En prenant $x=\e^{1/\sigma}$, ce qui est possible par hypothèse sur $\sigma$, nous obtenons l'inégalité suivante, pour $0\leqslant \vartheta<1$
$$S(\vartheta)\ll \Big(\frac{1}{\sigma+\vartheta}\Big)^{y\kappa(r)}\Big(1+\frac{\vartheta}{\sigma}\Big)^{(\rho+2)y}{\rm ,}$$
\noindent
et pour $\vartheta>1$
$$S(\vartheta)\ll \frac{1}{\sigma^{(\rho+2)y}}\log(2+\vartheta)^B {\rm .}$$
\noindent
Nous découpons l'intégrale au point $\vartheta=1$. Celle portant sur l'ensemble $\{\vartheta>1\}$ est convenablement majorée. Pour l'autre intégrale, nous avons
\begin{align*}
\int_{0}^{1}{S(\vartheta){\rm d}\vartheta} & \ll \int_{0}^{1}{\Big(\frac{1}{\sigma+\vartheta}\Big)^{\kappa(r)y}\Big(1+\frac{\vartheta}{\sigma}\Big)^{(\rho+2)y}{\rm d}\vartheta} \\
& \ll \frac{1}{\sigma^{\kappa(r)y}}\int_{0}^{1}{\Big(\frac{1}{1+\frac{\vartheta}{\sigma}}\Big)^{\kappa(r)y-(\rho+2)y}{\rm d}\vartheta}{\rm .}
\end{align*}
D'après la Proposition \ref{prop}, nous avons $\kappa(r)<3$. Un calcul direct de l'intégrale permet alors d'obtenir le majorant souhaité, en notant que, uniformément pour $\sigma \in ]0,1[$
$$\frac{1}{\sigma^{\kappa(r)y-1}}\log(1/\sigma)\ll \frac{1}{\sigma^{3y-1}}{\rm .}$$
\end{proof}

\begin{lemme}
\label{lemme 13}
Soient $\boldsymbol{\chi}$ un couple de deux caractères de Dirichlet non principaux tels que $\chi_1 \overline{ \chi_2}$ soit non principal, $A$, $c$ des constantes strictement positives et $\eta \in ]0,1[$ , $g \in \mathcal{M}_A(\boldsymbol{\chi},c,\eta)$ et $T>1$ un réel. Nous avons, uniformément pour $\sigma>0$ et $\vartheta \in \mathbb{R}$
$$\sum\limits_{n\geqslant 1}{\frac{\mu^2(n)g(n)}{a_nb_n^{1+\sigma}}\widehat{\tau_{1}}(n,\boldsymbol{\chi},\vartheta)} \ll 
\left\{
 \begin{array}{ll}
\frac{1}{\sigma^{3y}}  & \mbox{si }   |\vartheta|\leqslant 1 {\rm ,}\\
\frac{1}{\sigma^{3y}} \log(2+|\vartheta|)^B & \mbox{sinon,} \\
\end{array}
\right.$$
où $a_n$ et $b_n$ sont définis en \eqref{eq anbn} et $\widehat{\tau_{1}}(n,\boldsymbol{\chi},\vartheta)$ en \eqref{hattau1}.
\end{lemme}

\begin{proof}
La seule différence par rapport à la démonstration du Lemme~\ref{lemme 12} est l'utilisation du Lemme \ref{lemme 11} pour majorer le produit eulérien.
\end{proof}

\section{Estimation de $C(n,\boldsymbol{\chi})$ et $D(n,\boldsymbol{\chi})$}
Pour $0<\sigma\leqslant 1/10$, nous notons
\begin{equation}
\label{eq 1}
\mathcal{L}_1(\sigma):=\exp\Big(\sqrt{\log(1/\sigma)\log_2(1/\sigma)}\Big) {\rm .}
\end{equation}
Pour $\sigma> 0$, $g$ une fonction arithmétique et $T>1$, nous notons
$$\mathfrak{S}_1(\sigma,g,\boldsymbol{\chi}):=\sum\limits_{n\geqslant 1}{\frac{\mu^2(n)}{a_nb_n^{1+\sigma}}g(n)C(n,\boldsymbol{\chi})}{\rm ,}$$
\noindent
où $C(n,\boldsymbol{\chi})$ est défini en \eqref{eq C} et $a_n$ et $b_n$ sont définis en \eqref{eq anbn}. 
L'objectif de cette partie est de démontrer le théorème suivant.
\begin{theoreme}
\label{theo 3}
Soit $\boldsymbol{\chi}=(\chi_1, \chi_2)$ un couple de deux caractères de Dirichlet non principaux tels que $\chi_1\overline{\chi_2}$ soit non principal. Soient $A$, $c$ des constantes strictement positives, $\eta \in ]0,1[$  et $g \in \mathcal{M}_A(\boldsymbol{\chi},c,\eta)$. Si $0< y=y(g)$,  il existe une constante $\alpha>0$, dépendant au plus de $g$, $\boldsymbol{\chi}$, $c$ et $\eta$, telle que nous ayons, uniformément pour $\sigma>0$
$$\mathfrak{S}_1(\sigma,g,\boldsymbol{\chi})\ll\frac{1}{\sigma^{\max\{y+1,m(y,\rho)\}}}\mathcal{L}_1(\sigma)^{\alpha}{\rm ,}$$
où $m(y,\rho)$ est défini en \eqref{eq m}, et où nous avons posé
\begin{equation}
\label{eq T}
\log T:=\Big(\frac{9yq\log \frac{1}{\sigma}}{c}\Big)^{1/\eta}
\end{equation}
et 
\begin{equation}
\label{eq q}
q:=\sqrt{\frac{\log (1/\sigma)}{\log_2(1/\sigma)}}{\rm .}
\end{equation}
Nous avons le même résultat en remplaçant $C(n,\boldsymbol{\chi})$ par $D(n,\boldsymbol{\chi})$.
\end{theoreme}
\subsection{Estimation de $\Delta^{(k)}(n,\boldsymbol{\chi},\vartheta)$}
Pour $\sigma> 0$, $g$ une fonction arithmétique et $\vartheta \in \mathbb{R}$, nous notons
$$\mathfrak{S}^{(k)}(\sigma,g,\boldsymbol{\chi},\vartheta):=\sum\limits_{n\geqslant 1}{\frac{\mu^2(n)g(n)}{a_nb_n^{1+\sigma}}\Delta^{(k)}(n,\boldsymbol{\chi},\vartheta)^2}{\rm ,}$$
où $\Delta^{(k)}(n,\boldsymbol{\chi},\vartheta)$ est défini en \eqref{Dktheta} et  $a_n$ et $b_n$ sont définis en \eqref{eq anbn} à partir de la valeur de $T$ déterminée en \eqref{eq T}.
\begin{prop}
\label{prop 1}
Soient $\boldsymbol{\chi}=(\chi_1, \chi_2)$ un couple de deux caractères de Dirichlet non principaux tels que $\chi_1\overline{\chi_2}$ soit non principal. Soient $A$, $c$ des constantes strictement positives, $\eta \in ]0,1[$ et $g \in \mathcal{M}_A(\boldsymbol{\chi},c,\eta)$. Pour $0<y=y(g)$, il existe une constante $\alpha>0$, dépendant au plus de $g$, $\boldsymbol{\chi}$, $c$ et $\eta$, telle que nous ayons, uniformément pour $\sigma>0$, $k\leqslant \sqrt{\log 1/\sigma}$ entier et $|\vartheta| \leqslant q\sigma$ 
$$\mathfrak{S}^{(k)}(\sigma,g,\boldsymbol{\chi},\vartheta)\ll\frac{(k+1)^3}{\sigma^{3y}}\mathcal{L}_1(\sigma)^{\alpha}{\rm ,}$$
uniformément pour $\sigma>0$, $k\leqslant \sqrt{\log 1/\sigma}$ entier et $q\sigma\leqslant |\vartheta| \leqslant \frac{q}{\log T}$ 
$$\mathfrak{S}^{(k)}(\sigma,g,\boldsymbol{\chi},\vartheta)\ll\frac{(k+1)^3|\vartheta|^{m(y,\rho)-3y}}{\sigma^{m(y,\rho)}}\mathcal{L}_1(\sigma)^{\alpha}{\rm ,}$$
uniformément pour $\sigma>0$, $k\leqslant \sqrt{\log 1/\sigma}$ entier et $\frac{q}{\log T}\leqslant |\vartheta| \leqslant \e^{c(\log T)^{\eta}}$ 
$$\mathfrak{S}^{(k)}(\sigma,g,\boldsymbol{\chi},\vartheta)\ll\frac{(k+1)^3}{\sigma^{m(y,\rho)}}\mathcal{L}_1(\sigma)^{\alpha}\log(3+|\vartheta|)^B$$
et uniformément pour $\sigma>0$, $k\leqslant \sqrt{\log 1/\sigma}$ entier et $|\vartheta| \geqslant \e^{c(\log T)^{\eta}}$
$$\mathfrak{S}^{(k)}(\sigma,g,\boldsymbol{\chi},\vartheta)\ll\frac{(k+1)^3}{\sigma^{4y}}\mathcal{L}_1(\sigma)^{\alpha}\log(2+|\vartheta|)^B {\rm ,}$$
où $\rho$ est défini en \eqref{eq lambda}, $\mathcal{L}_1(\sigma)$ est défini en \eqref{eq 1}, et $T$ et $q$ sont définis en \eqref{eq T} et \eqref{eq q}.
\end{prop} 

\begin{proof}
Cette démonstration s'appuie sur la méthode différentielle de \cite{B}.\\

Nous commençons par utiliser le Lemme \ref{lemme 3} pour majorer $\Delta^{(k)}(n,\boldsymbol{\chi},\vartheta)$. Le membre de droite du majorant se traite au moyen du Lemme \ref{lemme 12}, qui fournit la majoration 
$$\sum\limits_{n\geqslant 1}{\frac{\mu^2(n)}{a_nb_n^{1+\sigma}}g(n)\max\limits_{d\mid n}|\tau(d,\chi_1,\vartheta)|^2}\ll \frac{1}{\sigma^{m(y,\rho)}}{\rm .}$$
Pour traiter le membre de gauche du majorant, nous utilisons dans un premier temps le Lemme~\ref{lemme 4}, qui nous permet de restreindre la somme aux entiers $n$ pour lesquels le facteur $E(n)$, défini en \eqref{E} n'est pas trop petit. Pour de tels entiers $n$, le terme $E(n)^{-4/q}$, où $q$ est défini en \eqref{eq q}, est uniformément majoré par $\mathcal{L}_1(\sigma)^{\alpha}$ pour une constante $\alpha$ bien choisie. Il s'agit donc étudier 
$$L^{(k)}_{\vartheta,T,q}(\sigma):= \sum\limits_{n\geqslant 1}{\frac{\mu^2(n)}{a_nb_n^{1+\sigma}}g(n)M^{(k)}_{2q}(n,\boldsymbol{\chi},\vartheta)^{1/q}} {\rm ,}$$
où $M^{(k)}_{2q}(n,\boldsymbol{\chi},\vartheta)$ est défini en en \eqref{eq Mijtheta}. En dérivant la fonction $L^{(k)}_{\vartheta,T,q}(\sigma)$, nous obtenons
\begin{align*}
-\big(L^{(k)}_{\vartheta,T,q}\big)'(s)&=\sum\limits_{n\geqslant 1}{\frac{\mu^2(n)}{a_nb_n^{1+s}}g(n)M^{(k)}_{2q}(n,\boldsymbol{\chi},\vartheta)^{1/q}\log b_n}\\
&=\sum\limits_{n\geqslant 1}{\frac{\mu^2(n)}{a_nb_n^{1+s}}g(n)M^{(k)}_{2q}(n,\boldsymbol{\chi},\vartheta)^{1/q}\sum\limits_{\substack{p\mid n \\ p>T}}{\log p}}\\
&= \sum\limits_{n\geqslant 1}{\frac{\mu^2(n)}{a_nb_n^{1+s}}g(n)\sum\limits_{\substack{(n,p)=1\\p>T}}{M^{(k)}_{2q}(np,\boldsymbol{\chi},\vartheta)^{1/q}\frac{g(p)\log p}{p^{1+s}}}}{\rm .}
\end{align*}
Un calcul explicite de la quantité $\Delta^{(k)}(np,\boldsymbol{\chi},\vartheta,u,v_1,v_2)$, valable pour tout $p$ premier et $n \in \mathbb{N}$ tel que $(n,p)=1$, fournit 
$$\Delta^{(k)}(np,\boldsymbol{\chi},\vartheta,u,v_1,v_2)=\big(1+\chi_1(p)p^{i\vartheta}\big)\Delta^{(k)}(n,\boldsymbol{\chi},\vartheta,u,v_1,v_2)+\chi_2(p)\Delta^{(k)}(n,\boldsymbol{\chi},\vartheta,u-\log p,v_1,v_2){\rm .}$$
En prenant le module à la puissance $2q$ de cette égalité, et en utilisant des inégalités classiques, nous obtenons
\begin{align*}
&\big\lvert\Delta^{(k)}(np,\boldsymbol{\chi},\vartheta,u,v_1,v_2)\big\rvert\\
\leqslant &2\big\lvert1+\chi_1(p)p^{i\vartheta}\big\rvert^{2q}\big\lvert\Delta^{(k)}(n,\boldsymbol{\chi},\vartheta,u,v_1,v_2)\big\rvert^{2q}\\
&+3\big\lvert\Delta^{(k)}(n,\boldsymbol{\chi},\vartheta,u-\log p,v_1,v_2)\big\rvert^{2q}\\
&+(1+q)4^q\sum\limits_{j=1}^{q-1}{\binom{2q}{2j}\big\lvert\Delta^{(k)}(n,\boldsymbol{\chi},\vartheta,u,v_1,v_2)\big\rvert^{2j}\big\lvert\Delta^{(k)}(n,\boldsymbol{\chi},\vartheta,u-\log p,v_1,v_2)\big\rvert^{2(q-j)}}{\rm .}
\end{align*}
En intégrant cette inégalité sur $\mathbb{R}\times [0,1]^2$, nous obtenons
$$M^{(k)}_{2q}(np,\chi,\vartheta)\leqslant \Big(2|1+\chi_1(p)p^{i\vartheta}|^{2q}+3\Big)M^{(k)}_{2q}(n,\chi,\vartheta)+W^{(k)}_{q}(n,p,\vartheta){\rm ,}$$
\noindent
où
$$W^{(k)}_{q}(n,p,\vartheta):=(1+q)4^q\sum\limits_{j=1}^{q-1}{\binom{2q}{2j}N^{(k)}_{j,q}(n,\boldsymbol{\chi},\vartheta ,\log p)} {\rm ,}$$
les $N^{(k)}_{j,q}(n,\boldsymbol{\chi},\vartheta ,\log p)$ étant définis en \eqref{Nijkq}. Des inégalités de Hölder ainsi que le Lemme \ref{lemme 7} fournissent
\begin{align*}
-\big(L^{(k)}_{\vartheta,T,q}\big)'(s)\leqslant &\frac{D(s)}{s}L^{(k)}_{\vartheta,T,q}(s)\\
&+\frac{A}{s^{1-1/q}}\Big(L^{(k)}_{\vartheta,T,q}(s)\Big)^{(q-2)/(q-1)}\Big((k+1)^2\sum\limits_{n\geqslant 1}{\frac{\mu^2(n)}{a_nb_n^{1+\sigma}}\widehat{\tau_{1}}(n,\boldsymbol{\chi},\vartheta)}\Big)^{1/(q-1)}\\
&+B_1\frac{(k+1)^{1/q}(\log T)^{9y}}{s^{9y+1-1/q}\e^{-c/q(\log T)^{\eta}}}
\end{align*}
où
$$D(s):=s\sum\limits_{p>T}{\frac{g(p)}{p^{1+s}}\big(\max(1,|1+\chi_1(p)p^{i\vartheta}|^2)+O(1/q)\big)\log p}{\rm ,}$$
$\widehat{\tau_{1}}(n,\boldsymbol{\chi},\vartheta)$ est défini en \eqref{hattau1} et $A$ et $B_1$ sont des constantes absolues.

Si $|\vartheta|\leqslant q\sigma$, le Lemme \ref{lemme 8} nous permet d'obtenir $D(s)\leqslant 3y+O(1/q)$ pour tout $s\geqslant \sigma$. De plus, par le Lemme~\ref{lemme 13}, nous avons
\begin{align*}
-\big(L^{(k)}_{\vartheta,T,q}\big)'(s)\leqslant &\frac{3y+a_1/q}{s}L^{(k)}_{\vartheta,T,q}(s)+A\frac{(k+1)^{2/(q-1)}\Big(L^{(k)}_{\vartheta,T,q}(s)\Big)^{(q-2)/(q-1)}}{s^{1+3y/(q-1)-1/q}}\\
&+B_1\frac{(\log T)^9}{s^{9y+1-1/q}\e^{-c(\log T)^{\eta}}}{\rm ,}
\end{align*}
où $a_1>0$ est une constante absolue.
\noindent
En posant 
$$\varepsilon:=B_1(\log T)^{9y}e^{-(\log T)^{\eta}c/q}, \; \; \beta:=9y-1/q {\rm ,}$$
\noindent 
nous obtenons
$$-\big(L^{(k)}_{\vartheta,T,q}\big)'(s)\leqslant \phi_1\big(s,L^{(k)}_{\vartheta,T,q}(s)\big){\rm ,}$$
\noindent
avec
$$\phi_1(s,x):=\frac{3y+a_1/q}{s}x+\frac{A(k+1)^{2/(q-1)}x^{(q-2)/(q-1)}}{s^{1-1/q+3y/(q-1)}}+\frac{\varepsilon}{s^{\beta+1}} {\rm .}$$
\noindent
Posons 
$$\gamma_1:=3y+\frac{b_1}{q}, \; \; X_1(s):=\frac{K_1}{s^{\gamma_1}}+\frac{\varepsilon}{s^{\beta}}$$
\noindent
où $b_1>a_1$ est une constante. Remarquons dans ce cas l'inégalité suivante
$$\gamma_1+1\geqslant1-\frac{1}{q}+\frac{3y}{q-1}+\frac{q-2}{q-1}\gamma_1 {\rm .}$$
\noindent
En supposant les inégalités
\begin{align*}
&\gamma_1>(3y+\frac{a_1}{q})+A(k+1)^{2/(q-1)}K_1^{-1/(q-1)}{\rm ,} \\
&\beta>3y+\frac{a_1}{q}+A(k+1)^{2/(q-1)}K_1^{-1/(q-1)}+1{\rm ,}
\end{align*}
\noindent
nous avons
\begin{align*}
-X_1'(s)=&\frac{K_1\gamma_1}{s^{\gamma_1+1}}+\frac{\epsilon\beta}{s^{\beta+1}} \\
\geqslant& \frac{K_1(3y+a_1/q)}{s^{\gamma_1+1}}+\frac{\varepsilon(3y+a_1/q)}{s^{\beta+1}}+\frac{A(k+1)^{2/(q-1)}K_1^{(q-2)/(q-1)}}{s^{\gamma_1+1}}\Big(1+\frac{\varepsilon s^{\gamma_1}}{K_1s^{\beta}}\Big)\\
&+\frac{\varepsilon}{s^{\beta+1}} \\
\geqslant& \frac{K_1(3y+a_1/q)}{s^{\gamma+1}}+\frac{\varepsilon(3y+a_1/q)}{s^{\beta+1}}\\
&+\frac{A(k+1)^{2/(q-1)}K_1^{(q-2)/(q-1)}}{s^{\gamma_1+1}}\Big(1+\frac{\varepsilon s^{\gamma_1}}{K_1s^{\beta}}\Big)^{(q-2)/(q-1)}+\frac{\varepsilon}{s^{\beta+1}}\\
\geqslant &\phi_1\big(s,X_1(s)\big){\rm .}
\end{align*}
\begin{remarque}
Les hypothèses sur $\gamma_1$ et $\beta$ sont vérifiées  dès que nous choisissons $K$ de l'ordre de $(k+1)^{2}(qC_2)^q$ pour une constante $C_2$ assez grande.
\end{remarque}
Comme tout ceci n'est valable que pour $s\leqslant \frac{1}{\log T}$, nous notons $s_0=\frac{1}{\log T}$. Remarquons dans ce cas que 
$$L^{(k)}_{\vartheta,T,q}(s)\leqslant \sum\limits_{n\geqslant 1}{\frac{\mu^2(n)}{a_nb_n^{1+s}}\tau_3(n)^{2}} \ll \frac{(\log T)^{9y}}{s^{9y}} {\rm .} $$
\noindent 
Ainsi,
$$L^{(k)}_{\vartheta,T,q}(s_0)\leqslant X_1(s_0)$$
\noindent
en choisissant $K=(k+1)^{2}(qC_2)^q(\log T)^{18y}$. Le lemme 70.2 de \cite{HT}  permet alors de montrer l'inégalité suivante
$$L^{(k)}_{\vartheta,T,q}(s)\leqslant \frac{(k+1)^{2}(qC_2)^q(\log T)^{18y}}{s^{3y+a_1/q}}+\frac{B_1(\log T)^{9y}}{\e^{(\log T)^{\eta}c/q}s^{9y-1/q}} $$
dès que $\sigma \leqslant s\leqslant s_0$. Les choix de $T$ \eqref{eq T} et de $q$ \eqref{eq q} impliquent alors
$$L^{(k)}_{\vartheta,T,q}(\sigma)\ll \frac{(k+1)^{2}}{\sigma^{3y}}\mathcal{L}_1(\sigma)^{\alpha} {\rm .}$$

Le cas $q\sigma\leqslant |\vartheta|\leqslant \frac{q}{\log T}$ est le plus délicat, puisque le majorant de $D(s)$ fourni par le Lemme \ref{lemme 8} change en fonction de la position de $s$ par rapport à $\frac{|\vartheta|}{q}$.\\

Pour $\frac{|\vartheta|}{q}\leqslant s \leqslant \frac{1}{\log T}$, nous avons toujours $D(s)\leqslant {3y+a_1/q}$, la suite de la démonstration est inchangée, et nous permet d'avoir
$$L^{(k)}_{\vartheta,T,q}\big(\frac{|\vartheta|}{q}\big) \ll \frac{(k+1)^{2}}{|\vartheta|^{3y}}\mathcal{L}_{1}(\sigma)^{\alpha}$$
car $|\vartheta|\geqslant \sigma$ et $\mathcal{L}_1$ est décroissante. Par ailleurs, le facteur $q^{3y}$ peut être absorbé par le terme $\mathcal{L}_1(\sigma)^{\alpha}$ quitte à augmenter $\alpha$.\\

Pour $s\leqslant \frac{|\vartheta|}{q}$, nous avons cette fois $D(s)\leqslant (\rho+2)y+a_1/q$. Nous modifions donc $\gamma_1$ et $K_1$ par rapport à précédemment, nous choisissons
$$\gamma_1:=m(y,\rho)+b_1/q{\rm ,}$$
de manière à vérifier
$$\gamma_1+1\geqslant1-\frac{1}{q}+\frac{3y}{q-1}+\frac{q-2}{q-1}\gamma_1 $$
et
$$\gamma_1> (\rho+2)y+a_1/q+A(k+1)^{2/(q-1)}K_1^{-1/(q-1)}{\rm .}$$
Par ailleurs, pour assurer l'inégalité 
$$L^{(k)}_{\vartheta,T,q}\big(\frac{|\vartheta|}{q}\big)\leqslant X_1(\frac{|\vartheta|}{q}) {\rm ,}$$
nous choisissons
$$K_1:=(k+1)^{2}|\vartheta|^{m(y,\rho)-3y}\mathcal{L}_1(\sigma)^{\alpha}{\rm .}$$
Nous obtenons ainsi
$$L^{(k)}_{\vartheta,T,q}(\sigma)\ll\frac{(k+1)^{2}|\vartheta|^{m(y,\rho)-3y}}{\sigma^{m(y,\rho)}}\mathcal{L}_1(\sigma)^{\alpha}{\rm .}$$ 

Si $\frac{q}{\log T}\leqslant |\vartheta| \leqslant \e^{c(\log T)^{\eta}}$, nous avons systématiquement $D(s)\leqslant (\rho+~2)y+~a_1/q$ et le Lemme \ref{lemme 13} fournit cette fois la majoration
$$\sum\limits_{n\geqslant 1}{\frac{\mu^2(n)g(n)}{a_nb_n^{1+s}}\widehat{\tau_1}(n,\boldsymbol{\chi},\vartheta)}\ll \frac{1}{s^{3y}}\log(3+|\vartheta|)^B{\rm .}$$ 
Nous choisissons ainsi
$$\gamma_1:=m(y,\rho)+\frac{b_1}{q}$$
 et 
 $$K_1:=(k+1)^{2}(qC_2)^q\log(3+|\vartheta|)^B$$
 pour obtenir
 $$L^{(k)}_{\vartheta,T,q}(\sigma)\ll \frac{(k+1)^{2}}{\sigma^{m(y,\rho)}}\mathcal{L}_1(\sigma)^{\alpha}{\rm .}$$
 
 Enfin, si $|\vartheta|\geqslant \e^{c(\log T)^{\eta}}$, nous avons $D_1(s)\leqslant 4y+\frac{a_1}{q}$. Nous choisissons ainsi
 $$\gamma_{1}:=4y+\frac{b_1}{q}$$
 et 
 $$K_1:=(k+1)^{2}(qC_2)^q\log(2+|\vartheta|)^B$$
 pour obtenir
 $$L^{(k)}_{\vartheta,T,q}(\sigma)\ll \frac{(k+1)^{2}}{\sigma^{4y}}\log(2+|\vartheta|)^B\mathcal{L}_1(\sigma)^{\alpha}{\rm .}$$
\end{proof} 
\subsection{Démonstration du Théorème \ref{theo 3}}
La Proposition \ref{prop 1} nous permet de déterminer la majoration du Théorème \ref{theo 3} pour $C(n,\boldsymbol{\chi})$, défini en \eqref{eq C}. En effet, d'après le théorème de Parseval, pour tout $(u_2,v_1,v_2) \in \mathbb{R}\times [0,1]$, nous avons
$$C(n,\boldsymbol{\chi},u_2,v_1,v_2)\ll \int_{\mathbb{R}}{|\Delta^{(0)}(n,\chi_2,\chi_1,\vartheta,u_2,v_2,v_1)|^2\frac{{\rm d}\vartheta}{1+\vartheta^2}}{\rm ,}$$
où $\Delta^{(0)}(n,\chi_2,\chi_1,\vartheta,u_2,v_2,v_1)$ est défini en \eqref{Dktheta}. Ainsi,
$$C(n,\boldsymbol{\chi})\ll \int_{\mathbb{R}}{|\Delta^{(0)}(n,\chi_2,\chi_1,\vartheta)|^2\frac{{\rm d}\vartheta}{1+\vartheta^2}}{\rm .}$$
Elle nous permet également de déterminer la majoration du Théorème \ref{theo 3} pour $D(n,\boldsymbol{\chi})$, défini en \eqref{eq D}. Cependant, les calculs sont plus complexes. Nous commençons par calculer, pour tout $v_1,v_2$ la transformée de Fourier de la fonction
$$(u_1,u_2)\mapsto \Delta_3(n,\boldsymbol{\chi},u_1,u_2,v_1,v_2)$$
que nous noterons
$$\widehat{\Delta}(n,\boldsymbol{\chi},\vartheta_1,\vartheta_2,v_1,v_2){\rm .}$$
Nous notons également pour la suite, pour $(v_1,v_2,\vartheta) \in \mathbb{R}^3$,
\begin{equation}
\label{k}
k(v_1,v_2,\vartheta):=\frac{\e^{i(v_1+v_2)\vartheta}-\e^{iv_1\vartheta}}{i\vartheta}{\rm .}
\end{equation}
Nous avons
\begin{align*}
\widehat{\Delta}(n,\boldsymbol{\chi},\vartheta_1,\vartheta_2,v_1,v_2)&=\int_{\mathbb{R}}{\int_{\mathbb{R}}{\Delta_3(n,\boldsymbol{\chi},u_1,u_2,v_1,v_2)\e^{-i\vartheta_1u_1}\e^{-i\vartheta_2u_2}{\rm d}u_1}{\rm d}u_2}\\
&=\sum\limits_{d_1d_2\mid n}{\chi_1(d_1)\chi_2(d_2)\int_{\log d_1-v_1}^{\log d_1}{\int_{\log d_2-v_2}^{\log d_2}{\e^{-i\vartheta_1u_1}\e^{-i\vartheta_2u_2}{\rm d}u_1}{\rm d}u_2}}\\
&=\tau(n,\boldsymbol{\chi},-\vartheta_1,-\vartheta_2)k(0,v_1,\vartheta_1)k(0,v_2,\vartheta_2){\rm ,}
\end{align*}
où $\tau(n,\boldsymbol{\chi},-\vartheta_1,-\vartheta_2)$ est défini en \eqref{tautheta12}. Ainsi, pour tout $(u_1,u_2)$ sauf un nombre fini, nous avons
\begin{align*}
&\Delta_3(n,\boldsymbol{\chi},u_1,u_2,v_1,v_2)\\
=&\frac{1}{4\pi^2}\int_{\mathbb{R}}{\int_{\mathbb{R}}{\tau(n,\boldsymbol{\chi},-\vartheta_1,-\vartheta_2)k(0,v_1,\vartheta_1)k(0,v_2,\vartheta_2)\e^{i\vartheta_1u_1}\e^{i\vartheta_2u_2}{\rm d}\vartheta_1}{\rm d}\vartheta_2}
\end{align*}
En remplaçant $u_1$ par $u-u_2$, puis en effectuant le changement de variable $\vartheta_2=\vartheta_2-\vartheta_1$, nous obtenons
\begin{align}
\label{hatdelta3}
\begin{split}
&\Delta_3(n,\boldsymbol{\chi},u-u_2,u_2,v_1,v_2)\\
=&\frac{1}{4\pi^2}\int_{\mathbb{R}}{\int_{\mathbb{R}}{\tau(n,\boldsymbol{\chi},-\vartheta_1,-\vartheta_2)k(0,v_1,\vartheta_1)k(0,v_2,\vartheta_2)\e^{i\vartheta_1(u-u_2)}\e^{i\vartheta_2u_2}{\rm d}\vartheta_1}{\rm d}\vartheta_2}\\
=&\frac{1}{4\pi^2}\int_{\mathbb{R}}{\int_{\mathbb{R}}{\tau(n,\boldsymbol{\chi},-\vartheta_1,-\vartheta_2)k(0,v_1,\vartheta_1)k(0,v_2,\vartheta_2)\e^{i\vartheta_1u}\e^{i(\vartheta_2-\vartheta_1)u_2}{\rm d}\vartheta_1}{\rm d}\vartheta_2}\\
=&\frac{1}{4\pi^2}\int_{\mathbb{R}}{\int_{\mathbb{R}}{\tau(n,\boldsymbol{\chi},-\vartheta_1,-\vartheta_2-\vartheta_1)k(0,v_1,\vartheta_1)k(0,v_2,\vartheta_1+\vartheta_2)\e^{i\vartheta_1u}\e^{i\vartheta_2u_2}{\rm d}\vartheta_1}{\rm d}\vartheta_2}{\rm .}
\end{split}
\end{align}
Posons
\begin{align}
\begin{split}
\label{F}
&F(n,\boldsymbol{\chi},u,v_1,v_2,\vartheta_2)\\
:=&\int_{\mathbb{R}}{\tau(n,\boldsymbol{\chi},-\vartheta_1,-\vartheta_2-\vartheta_1)k(0,v_1,\vartheta_1)k(0,v_2,\vartheta_1+\vartheta_2)\e^{i\vartheta_1u}{\rm d}\vartheta_1}{\rm .}
\end{split}
\end{align}
Le théorème de Parseval nous permet de relier la valeur de $D(n,\boldsymbol{\chi},u,v_1,v_2)$ au moment $L^2$  de $F(n,\boldsymbol{\chi},u,v_1,v_2,\vartheta_2)$. Nous transformons l'expression de $F$ afin de relier cette fonction à des fonctions $\Delta$ de type \eqref{Dktheta} et pouvoir utiliser la Proposition \ref{prop 1} pour estimer $D(n,\boldsymbol{\chi})$. Pour cela, nous linéarisons en $\vartheta_1$ la quantité $k(0,v_1,\vartheta_1)k(0,v_2,\vartheta_1+\vartheta_2)$. Nous commençons par utiliser la décomposition en éléments simples,
\begin{equation}
\label{DES}
\frac{1}{\vartheta_1(\vartheta_2+\vartheta_1)}=\frac{1}{\vartheta_1\vartheta_2}-\frac{1}{\vartheta_2(\vartheta_2+\vartheta_1)}{\rm .}
\end{equation}
Par ailleurs, nous avons
\begin{align}
\begin{split}
\label{e1}
&(\e^{iv_1\vartheta_1}-1)(\e^{iv_2(\vartheta_1+\vartheta_2)}-1)\\
=&(\e^{iv_1\vartheta_1}-1)(\e^{iv_2(\vartheta_1+\vartheta_2)}-\e^{iv_2\vartheta_1}+\e^{iv_2\vartheta_1}-1)\\
=&(\e^{iv_2\vartheta_2}-1)(\e^{i(v_1+v_2)\vartheta_1}-\e^{iv_2\vartheta_1})+(\e^{iv_2\vartheta_1}-1)(\e^{iv_1\vartheta_1}-1)
\end{split}
\end{align}
et
\begin{align}
\begin{split}
\label{e2}
&(\e^{iv_1\vartheta_1}-1)(\e^{iv_2(\vartheta_1+\vartheta_2)}-1)\\
=&(\e^{iv_1((\vartheta_1+\vartheta_2)-\vartheta_2)}-1)(\e^{iv_2(\vartheta_1+\vartheta_2)}-1)\\
=&(\e^{iv_1((\vartheta_1+\vartheta_2)-\vartheta_2)}-\e^{-iv_1\vartheta_2}+\e^{-iv_1\vartheta_2}-1)(\e^{iv_2(\vartheta_1+\vartheta_2)}-1)\\
=&\e^{-iv_1\vartheta_2}(\e^{iv_1(\vartheta_1+\vartheta_2)}-1)(\e^{iv_2(\vartheta_1+\vartheta_2)}-1)+(\e^{-iv_1\vartheta_2}-1)(\e^{iv_2(\vartheta_1+\vartheta_2)}-1)
\end{split}
\end{align}
Ainsi, d'après \eqref{DES}, \eqref{e1} et \eqref{e2}, nous avons
\begin{align}
\label{k1234}
\begin{split}
&k(0,v_1,\vartheta_1)k(0,v_2,\vartheta_1+\vartheta_2)\\
=&\frac{1}{i\vartheta_2}\frac{(\e^{iv_1\vartheta_1}-1)(\e^{iv_2(\vartheta_1+\vartheta_2)}-1)}{i\vartheta_1}\\
&-\frac{1}{i\vartheta_2}\frac{(\e^{iv_1\vartheta_1}-1)(\e^{iv_2(\vartheta_1+\vartheta_2)}-1)}{i(\vartheta_1+\vartheta_2)}\\
=&\frac{\e^{iv_2\vartheta_2}-1}{i\vartheta_2}\frac{(\e^{i(v_1+v_2)\vartheta_1}-\e^{iv_2\vartheta_1})}{i\vartheta_1}+\frac{1}{i\vartheta_2}\frac{(\e^{iv_1\vartheta_1}-1)(\e^{iv_2\vartheta_1}-1)}{i\vartheta_1}\\
&-\frac{\e^{-iv_1\vartheta_2}}{i\vartheta_2}\frac{(\e^{iv_1(\vartheta_1+\vartheta_2)}-1)(\e^{iv_2(\vartheta_1+\vartheta_2)}-1)}{i(\vartheta_1+\vartheta_2)}-\frac{\e^{-iv_1\vartheta_2}-1}{i\vartheta_2}\frac{\e^{iv_2(\vartheta_1+\vartheta_2)}-1}{i(\vartheta_1+\vartheta_2)}\\
=&k(0,v_2,\vartheta_2)k(v_2,v_1,\vartheta_1)+\frac{1}{i\vartheta_2}\big(k(v_1,v_2,\vartheta_1)-k(0,v_2,\vartheta_1)\big)\\
&+k(-v_1,v_1,\vartheta_2)k(0,v_2,\vartheta_1+\vartheta_2)-\frac{\e^{-iv_1\vartheta_2}}{i\vartheta_2}(k(v_1,v_2,\vartheta_1+\vartheta_2)-k(0,v_2,\vartheta_1+\vartheta_2))\\
=&k(0,v_2,\vartheta_2)k(v_2,v_1,\vartheta_1)+k(-v_1,v_1,\vartheta_2)k(v_1,v_2,\vartheta_1+\vartheta_2)\\
&+\frac{1}{i\vartheta_2}\big(k(v_1,v_2,\vartheta_1)-k(v_1,v_2,\vartheta_1+\vartheta_2)\e^{-iv_1\vartheta_2}\big)\\
&+\frac{1}{i\vartheta_2}\big(k(0,v_2,\vartheta_1+\vartheta_2)\e^{-iv_1\vartheta_2}-k(0,v_2,\vartheta_1)\big){\rm .}
\end{split}
\end{align}
Pour $(v_1,v_2,\vartheta_1,\vartheta_2) \in [0,1]^2\times \mathbb{R}^2$, nous définissons
\begin{align}
k_1(v_1,v_2,\vartheta_1,\vartheta_2)&:=k(v_2,v_1,\vartheta_1)\\
k_2(v_1,v_2,\vartheta_1,\vartheta_2)&:=k(v_1,v_2,\vartheta_1+\vartheta_2)\\
k_3(v_1,v_2,\vartheta_1,\vartheta_2)&:=k(v_1,v_2,\vartheta_1)-k(v_1,v_2,\vartheta_1+\vartheta_2)\e^{-iv_1\vartheta_2}\\
k_4(v_1,v_2,\vartheta_1,\vartheta_2)&:=k(0,v_2,\vartheta_1+\vartheta_2)\e^{-iv_1\vartheta_2}-k(0,v_2,\vartheta_1){\rm ,}
\end{align}
et pour $j=1,\ldots 4$, nous posons
\begin{align}
\begin{split}
\label{F1234}
&F_j(n,\boldsymbol{\chi},u,v_1,v_2,\vartheta_2)\\
:=&\int_{\mathbb{R}}{\tau(n,\boldsymbol{\chi},-\vartheta_1,-\vartheta_2-\vartheta_1)k_j(v_1,v_2,\vartheta_1,\vartheta_2)\e^{i\vartheta_1u}{\rm d}\vartheta_1}{\rm .}
\end{split}
\end{align}
En injectant les équations \eqref{k1234} à \eqref{F1234} dans \eqref{F}, nous obtenons
\begin{align*}
&F(n,\boldsymbol{\chi},u,v_1,v_2,\vartheta_2)\\
=&k(0,v_2,\vartheta_2)F_1(n,\boldsymbol{\chi},u,v_1,v_2,\vartheta_2)+k(-v_1,v_1,\vartheta_2)F_2(n,\boldsymbol{\chi},u,v_1,v_2,\vartheta_2)\\
&+\frac{1}{i\vartheta_2}F_3(n,\boldsymbol{\chi},u,v_1,v_2,\vartheta_2)+\frac{1}{i\vartheta_2}F_4(n,\boldsymbol{\chi},u,v_1,v_2,\vartheta_2){\rm .}
\end{align*}
D'après le théorème de Parseval et l'égalité \eqref{hatdelta3}, nous avons
$$D(n,\boldsymbol{\chi},u,v_1,v_2)\ll\int_{\mathbb{R}}{\lvert F(n,\boldsymbol{\chi},u,v_1,v_2,\vartheta_2)\rvert^{2}{\rm d}\vartheta_2}{\rm ,}$$
où $D(n,\boldsymbol{\chi},u,v_1,v_2)$ est défini en \eqref{eq Du}. En injectant \eqref{F1234} dans cette inégalité, nous obtenons
\begin{align}
\label{majD}
\begin{split}
D(n,\boldsymbol{\chi})\ll &\int_{\mathbb{R}}{\frac{1}{1+\vartheta_2^2}\sup\limits_{\substack{u \in \mathbb{R}\\ (v_1,v_2)\in [0,1]^2}}\lvert F_1(n,\boldsymbol{\chi},u,v_1,v_2,\vartheta_2)\rvert^{2}{\rm d}\vartheta_2}\\
&+\int_{\mathbb{R}}{\frac{1}{1+\vartheta_2^2}\sup\limits_{\substack{u \in \mathbb{R}\\ (v_1,v_2)\in [0,1]^2}}\lvert F_2(n,\boldsymbol{\chi},u,v_1,v_2,\vartheta_2)\rvert^{2}{\rm d}\vartheta_2}\\
&+\int_{\mathbb{R}}{\frac{1}{\vartheta_2^2}\sup\limits_{\substack{u \in \mathbb{R}\\ (v_1,v_2)\in [0,1]^2}}\lvert F_3(n,\boldsymbol{\chi},u,v_1,v_2,\vartheta_2)\rvert^{2}{\rm d}\vartheta_2}\\
&+\int_{\mathbb{R}}{\frac{1}{\vartheta_2^2}\sup\limits_{\substack{u \in \mathbb{R}\\ (v_1,v_2)\in [0,1]^2}}\lvert F_4(n,\boldsymbol{\chi},u,v_1,v_2,\vartheta_2)\rvert^{2}{\rm d}\vartheta_2}{\rm .}
\end{split}
\end{align}
Nous voyons ainsi que l'estimation de
$$\sum\limits_{n\geqslant 1}{\frac{\mu^2(n)g(n)}{n^{1+\sigma}}D(n,\boldsymbol{\chi})}$$
se déduit de celle de
\begin{equation}
\label{Fi}
\sum\limits_{n\geqslant 1}{\frac{\mu^2(n)g(n)}{n^{1+\sigma}}\sup\limits_{u,v_1,v_2}\lvert F_j(n,\boldsymbol{\chi},u,v_1,v_2,\vartheta_2)\rvert^{2}}
\end{equation}
pour $j=1,\ldots,4$. Il nous faut maintenant relier les fonctions $F_j$ aux fonctions $\Delta$ de type \eqref{Dktheta}. Pour cela, nous constatons que nous avons, pour $(v_1,v_2,\vartheta_2)\in ~[0,1]^2\times ~\mathbb{R}$ et tout $u \in \mathbb{R}$ sauf un nombre fini, d'après la formule d'inversion de Fourier,
\begin{align*}
&\int_{\mathbb{R}}{\tau(n,\boldsymbol{\chi},-\vartheta_1,-\vartheta_2-\vartheta_1)k(v_1,v_2,\vartheta_1)\e^{i\vartheta_1u}{\rm d}\vartheta_1}\\ 
=&\int_{\mathbb{R}}{\sum\limits_{\substack{d_1d_2\mid n}}{\chi_1(d_1)\chi_2(d_2)d_2^{-i\vartheta_2}(d_1d_2)^{-i\vartheta_1}}k(v_1,v_2,\vartheta_1)\e^{i\vartheta_1u}{\rm d}\vartheta_1}\\
=&\sum\limits_{\substack{d_1d_2\mid n}}{\chi_1(d_1)\chi_2(d_2)d_2^{-i\vartheta_2}\int_{\mathbb{R}}{(d_1d_2)^{-i\vartheta_1}}k(v_1,v_2,\vartheta_1)\e^{i\vartheta_1u}{\rm d}\vartheta_1}\\
=&\sum\limits_{\substack{d_1d_2\mid n\\ \e^{u+v_1}\leqslant d_1d_2<\e^{u+v_1+v_2}}}{\chi_1(d_1)\chi_2(d_2)d_2^{-i\vartheta_2}}{\rm ,}
\end{align*}
où $\tau(n,\boldsymbol{\chi},\vartheta_1,\vartheta_2)$ est défini en \eqref{tautheta12} et $k(v_1,v_2,\vartheta_1)$ en \eqref{k}. De même, en effectuant le changement de variable $\vartheta_1=\vartheta_1+\vartheta_2$, nous avons
\begin{align*}
&\int_{\mathbb{R}}{\tau(n,\boldsymbol{\chi},-\vartheta_1,-\vartheta_2-\vartheta_1)k(v_1,v_2,\vartheta_1+\vartheta_2)\e^{i\vartheta_1u}{\rm d}\vartheta_1}\\
=&\int_{\mathbb{R}}{\tau(n,\boldsymbol{\chi},\vartheta_2-\vartheta_1,-\vartheta_1)k(v_1,v_2,\vartheta_1)\e^{i(\vartheta_1-\vartheta_2)u}{\rm d}\vartheta_1}\\
=&\sum\limits_{\substack{d_1d_2\mid n\\ \e^{u+v_1}\leqslant d_1d_2<\e^{u+v_1+v_2}}}{\chi_1(d_1)\chi_2(d_2)d_1^{i\vartheta_2}}\e^{-iu\vartheta_2}{\rm .}
\end{align*}
En injectant les deux dernières égalités, ainsi que la formule \eqref{k1234} dans \eqref{F1234}, nous obtenons
\begin{align}
\label{F1}
F_1(n,\boldsymbol{\chi},u,v_1,v_2,\vartheta_2)&=\!\!\!\!\!\!\!\!\sum\limits_{\substack{d_1d_2\mid n\\ \e^{u+v_2}\leqslant d_1d_2<\e^{u+v_1+v_2}}}{\!\!\!\!\!\!\!\!\chi_1(d_1)\chi_2(d_2)d_2^{-i\vartheta_2}}\\
\label{F2}
F_2(n,\boldsymbol{\chi},u,v_1,v_2,\vartheta_2)&=\!\!\!\!\!\!\!\!\sum\limits_{\substack{d_1d_2\mid n\\ \e^{u+v_1}\leqslant d_1d_2<\e^{u+v_1+v_2}}}{\!\!\!\!\!\!\!\!\chi_1(d_1)\chi_2(d_2)d_1^{i\vartheta_2}\e^{-iu\vartheta_2}}\\
\label{F3}
F_3(n,\boldsymbol{\chi},u,v_1,v_2,\vartheta_2)&=\!\!\!\!\!\!\!\!\sum\limits_{\substack{d_1d_2\mid n\\ \e^{u+v_1}\leqslant d_1d_2<\e^{u+v_1+v_2}}}{\!\!\!\!\!\!\!\!\chi_1(d_1)\chi_2(d_2)d_2^{-i\vartheta_2}\big(1-(d_1d_2)^{i\vartheta_2}\e^{-i(u+v_1)\vartheta_2}\big)}\\
\label{F4}
F_4(n,\boldsymbol{\chi},u,v_1,v_2,\vartheta_2)&=\sum\limits_{\substack{d_1d_2\mid n\\ \e^{u}\leqslant d_1d_2<\e^{u+v_2}}}{\chi_1(d_1)\chi_2(d_2)d_2^{-i\vartheta_2}\big((d_1d_2)^{i\vartheta_2}\e^{-i(u+v_1)\vartheta_2}-1\big)}{\rm .}
\end{align}
Dans le cas où $u$ est tel que l'une des égalités ci-dessus n'est pas vérifiée, nous pouvons approcher $D(n,\boldsymbol{\chi},u,v_1,v_2)$ par $D(n,\boldsymbol{\chi},u+\varepsilon,v_1,v_2)$ où $\varepsilon>0$ est suffisamment petit et tel que les égalités \eqref{F1} à \eqref{F4} soient vérifiées pour $u+\varepsilon$. L'erreur que l'on commet est alors dominée par $\log (n) \big(\Delta^*(n,\chi_1)^2+\Delta^*(n,\chi_2)^2\big)$ d'après \eqref{eq Du}, la quantité $\Delta^*(n,\chi)$ étant définie en \eqref{eq lemme 1.2}. La majoration \eqref{majD} se réecrit donc
\begin{align}
\label{majD'}
\begin{split}
D(n,\boldsymbol{\chi})\ll &\int_{\mathbb{R}}{\frac{1}{1+\vartheta_2^2}\sup\limits_{\substack{u \in \mathbb{R}\\ (v_1,v_2)\in [0,1]^2}}\lvert F_1(n,\boldsymbol{\chi},u,v_1,v_2,\vartheta_2)\rvert^{2}{\rm d}\vartheta_2}\\
&+\int_{\mathbb{R}}{\frac{1}{1+\vartheta_2^2}\sup\limits_{\substack{u \in \mathbb{R}\\ (v_1,v_2)\in [0,1]^2}}\lvert F_2(n,\boldsymbol{\chi},u,v_1,v_2,\vartheta_2)\rvert^{2}{\rm d}\vartheta_2}\\
&+\int_{\mathbb{R}}{\frac{1}{\vartheta_2^2}\sup\limits_{\substack{u \in \mathbb{R}\\ (v_1,v_2)\in [0,1]^2}}\lvert F_3(n,\boldsymbol{\chi},u,v_1,v_2,\vartheta_2)\rvert^{2}{\rm d}\vartheta_2}\\
&+\int_{\mathbb{R}}{\frac{1}{\vartheta_2^2}\sup\limits_{\substack{u \in \mathbb{R}\\ (v_1,v_2)\in [0,1]^2}}\lvert F_4(n,\boldsymbol{\chi},u,v_1,v_2,\vartheta_2)\rvert^{2}{\rm d}\vartheta_2}\\
&+\log n \big(\Delta^*(n,\chi_1)^2+\Delta^*(n,\chi_2)^2\big){\rm ,}
\end{split}
\end{align}
où les fonctions $F_j$ sont désormais définies par les équations \eqref{F1} à \eqref{F4}.\\

L'estimation de la somme \eqref{Fi} pour $j=1,2$ se déduit alors de la Proposition \ref{prop 1}. Nous ne traitons ici que le cas $j=1$, le cas $j=2$ étant identique. Dans ce cas, nous pouvons utiliser la remarque \ref{rq prem} afin de ne considérer que les entiers $n$ qui sont premiers avec $q_1q_2$, où $q_i$ désigne le conducteur de $\chi_i$ pour $i=1,2$. Nous pouvons donc poser $d_3=n/(d_1d_2)$ dans la somme \eqref{F1} afin d'obtenir
\begin{align*}
\sum\limits_{\substack{d_1d_2\mid n\\ \e^{u+v_2}\leqslant d_1d_2< \e^{u+v_1+v_2}}}{\!\!\!\!\!\!\!\!\!\chi_1(d_1)\chi_2(d_2)d_2^{-i\vartheta_2}}&=\!\!\!\!\!\!\!\!\!\sum\limits_{\substack{d_2d_3\mid n\\ n\e^{-u-v_1-v_2}< d_3\leqslant n\e^{-u-v_2}}}{\!\!\!\!\!\!\!\!\!\chi_1\Big(\frac{n}{d_2d_3}\Big)\chi_2(d_2)d_2^{-i\vartheta_2}}\\
&=\chi_1(n)\!\!\!\!\!\!\!\!\!\sum\limits_{\substack{d_2d_3\mid n\\ n\e^{-u-v_1-v_2}< d_3\leqslant n\e^{-u-v_2}}}{\!\!\!\!\!\!\!\!\!\overline{\chi_1}(d_3)\overline{\chi_1}\chi_2(d_2)d_2^{-i\vartheta_2}}{\rm .}
\end{align*} 
Nous avons donc, uniformément pour $u,\vartheta_2,v_1,v_2 \in \mathbb{R}^2\times [0,1]^2$
$$\sum\limits_{n\geqslant 1}{\frac{\mu^2(n)g(n)}{n^{1+\sigma}}\sup\limits_{u,v_1,v_2}\lvert F_1(u,v_1,v_2,\vartheta_2)\rvert^{2}}\leqslant \sum\limits_{n\geqslant 1}{\frac{\mu^2(n)g(n)}{n^{1+\sigma}}\lvert \Delta_1^{(0)}(n,\overline{\chi_1},\overline{\chi_1}\chi_2,-\vartheta_2)\rvert^{2}}{\rm .}$$
La Proposition \ref{prop 1}, appliquée à $k=0$, fournit alors
\begin{equation}
\label{majF1}
\int_{\mathbb{R}}{\sum\limits_{n\geqslant 1}{\frac{\mu^2(n)g(n)}{n^{1+\sigma}}\sup\limits_{u,v_1,v_2}\lvert F_1(u,v_1,v_2,\vartheta_2)\rvert^{2}}\frac{{\rm d}\vartheta}{1+\vartheta_2^2}}\ll \frac{1}{\sigma^{m(y,\rho)}}\mathcal{L}_1(\sigma)^{\alpha}{\rm ,}
\end{equation}
où $m(y,\rho)$ est défini en \eqref{eq m} et $\mathcal{L}_1(\sigma)$ en \eqref{eq 1}.

Pour $j=3,4$, une difficulté apparaît, du fait que l'intégrale de $\frac{1}{\vartheta_2^2}$ est divergente en $0$. Nous ne détaillons que le cas $j=3$, le cas $j=4$ étant analogue. Nous posons alors $\vartheta_0:=\mathcal{L}_1(\sigma)^{-1}$. Pour $\lvert\vartheta_2\rvert\geqslant \vartheta_0$ et $j=3$, nous écrivons
\begin{align*}
\big\lvert F_3(u,v_1,v_2,\vartheta_2)\big\rvert^2\ll &\frac{1}{\lvert\vartheta_2\rvert^2}\Big\lvert\sum\limits_{\substack{d_1d_2\mid n\\ \e^{u+v_1}\leqslant d_1d_2< \e^{u+v_1+v_2}}}{\chi_1(d_1)\chi_2(d_2)d_2^{-i\vartheta_2}}\Big\rvert^2\\
&+\frac{1}{\lvert\vartheta_2\rvert^2}\Big\lvert\sum\limits_{\substack{d_1d_2\mid n\\ \e^{u+v_1}\leqslant d_1d_2< \e^{u+v_1+v_2}}}{\chi_1(d_1)\chi_2(d_2)d_1^{i\vartheta_2}}\Big\rvert^2 {\rm .}
\end{align*}

La Proposition \ref{prop 1}, appliquée à $k=0$, fournit alors
\begin{equation}
\label{majF3.1}
\int_{|\vartheta|>\vartheta_0}{\sum\limits_{n\geqslant 1}{\frac{\mu^2(n)g(n)}{n^{1+\sigma}}\sup\limits_{u,v_1,v_2}\lvert F_1(u,v_1,v_2,\vartheta_2)\rvert^{2}}\frac{{\rm d}\vartheta_2}{\vartheta_2^2}}\ll \frac{1}{\sigma^{m(y,\rho)}}\mathcal{L}_1(\sigma)^{\alpha+1}{\rm .}
\end{equation}

Il nous reste à traiter le cas $|\vartheta_2|\leqslant \vartheta_0$. Pour cela, nous effectuons alors un développement limité de 
$$1-\e^{-i(u+v_1)\vartheta_2}(d_1d_2)^{i\vartheta_2}$$
par rapport à $\vartheta_2$. Nous obtenons
\begin{align*}
1-\e^{-i(u+v_1)\vartheta_2}(d_1d_2)^{i\vartheta_2}=-\sum\limits_{k=1}^{k_0}{\frac{\vartheta_2^k}{k!}\Big(\log (d_1d_2)-u-v_1\Big)^k}+O\Big(\frac{\lvert \vartheta_2\rvert^{k_0+1}}{(k_0+1)!}\Big){\rm .}
\end{align*}
Ce qui implique, d'après \eqref{F3},
\begin{align}
\label{CS}
\begin{split}
&\frac{1}{\vartheta_2^2}\big\lvert F_3(u,v_1,v_2,\vartheta_2)\big\rvert^2 \\
\leqslant & \frac{1}{\vartheta_2^2} \sum\limits_{k=1}^{k_0}{\frac{1}{k!}}\sum\limits_{k=1}^{k_0}{\frac{\vartheta_2^{2k}}{k!}\Big\rvert \sum\limits_{\substack{d_1d_2\mid n\\ \e^{u+v_1}\leqslant d_1d_2< \e^{u+v_1+v_2}}}{\chi_1(d_1)\chi_2(d_2)d_2^{-i\vartheta_2}\Big(\log (d_1d_2)-u-v_1\Big)^{k}}\Big\rvert^2}\\
&+O\Big(\frac{\vartheta_2^{2(k_0-1)}}{(k_0+1)!^2}\tau_3(n)^2\Big)\\
\leqslant & \e \sum\limits_{k=1}^{k_0}{\frac{\vartheta_2^{2(k-1)}}{k!}\Big\lvert \!\!\!\!\!\!\!\!\!\!\!\! \sum\limits_{\substack{d_2d_3\mid n\\ n\e^{-u-v_1-v_2}< d_3\leqslant n\e^{-u-v_1}}}{\!\!\!\!\!\!\overline{\chi_1}(d_3)\overline{\chi_1}\chi_2(d_2)d_2^{-i\vartheta_2}\Big(\log n-u-v_1-\log d_3\Big)^{k}}\Big\rvert^2}\\
&+O\Big(\frac{\vartheta_2^{2(k_0-1)}}{(k_0+1)!^2}\tau_3(n)^2\Big)\\
\leqslant & \e \sum\limits_{k=1}^{k_0}{\frac{\vartheta_2^{2(k-1)}}{k!}\Big\lvert \Delta^{(k)}(n,\overline{\chi_1},\chi_2\overline{\chi_1},-\vartheta_2)\Big\rvert^2}+O\Big(\frac{\vartheta_2^{2(k_0-1)}}{(k_0+1)!^2}\tau_3(n)^2\Big){\rm .}
\end{split}
\end{align}

La Proposition \ref{prop 1} fournit alors , pour tout $1\leqslant k\leqslant k_0$, 
\begin{equation}
\label{majF3.2.1}
\int_{-\vartheta_0}^{\vartheta_0}{\sum\limits_{n\geqslant 1}{\frac{\mu^2(n)g(n)}{n^{1+\sigma}}\sup\limits_{u,v_1,v_2}\Big\lvert \Delta^{(k)}(n,\overline{\chi_1},\chi_2\overline{\chi_1},-\vartheta_2)\Big\rvert^2}\frac{\vartheta_2^{2(k-1)}}{k!}{\rm d}\vartheta}\ll \frac{(k+1)^3}{k! \sigma^{m(y,\rho)}}\mathcal{L}_1(\sigma)^{\alpha}{\rm .}
\end{equation}

Nous choisissons ensuite $k_0:=[5y \sqrt{\frac{\log 1/\sigma}{\log_2 1/\sigma}}]+2$, de sorte que 
\begin{equation}
\label{majF3.2.2}
\sum\limits_{n\geqslant 1}{\frac{\mu^2(n)g(n)}{n^{1+\sigma}}\tau_3(n)^2}\frac{1}{(k_0+1)!^2}\int_{-\vartheta_0}^{\vartheta_0}{\vartheta^{2(k_0-1)}{\rm d}\vartheta}\ll 1{\rm .}
\end{equation} 
En combinant les majorations \eqref{majF3.2.1} et \eqref{majF3.2.2} dans l'inégalité \eqref{CS}, nous obtenons
\begin{equation}
\label{majF3.2}
\int_{-\vartheta_0}^{\vartheta_0}{\sum\limits_{n\geqslant 1}{\frac{\mu^2(n)g(n)}{n^{1+\sigma}}\sup\limits_{u,v_1,v_2}\lvert F_3(u,v_1,v_2,\vartheta_2)\rvert^{2}}\frac{{\rm d}\vartheta_2}{\vartheta_2^2}}\ll \frac{1}{\sigma^{m(y,\rho)}}\mathcal{L}_1(\sigma)^{\alpha+1}{\rm .}
\end{equation}
Les équations \eqref{majF3.1} et \eqref{majF3.2} fournissent alors
\begin{equation}
\label{majF3}
\int_{\mathbb{R}}{\sum\limits_{n\geqslant 1}{\frac{\mu^2(n)g(n)}{n^{1+\sigma}}\sup\limits_{u,v_1,v_2}\lvert F_3(u,v_1,v_2,\vartheta_2)\rvert^{2}}\frac{{\rm d}\vartheta_2}{\vartheta_2^2}}\ll \frac{1}{\sigma^{m(y,\rho)}}\mathcal{L}_1(\sigma)^{\alpha+1}{\rm .}
\end{equation}
Enfin, le Théorème \ref{theo 2}, qui sera énoncé et démontré dans la section suivante, fournit la majoration
\begin{equation}
\label{majerr}
\sum\limits_{n\geqslant 1}{\frac{\mu^2(n)}{n^{1+\sigma}}g(n)\log n\Delta^*(n,\chi)^2}\ll \frac{1}{\sigma^{\max\{y+1,m(y,\rho)\}}}\mathcal{L}_1(\sigma)^{\alpha}{\rm ,}
\end{equation}
valable lorsque $\chi$ est un caractère de Dirichlet non principal.\\

En combinant les équations \eqref{majF1}, \eqref{majF3} et \eqref{majerr} à l'inégalité \eqref{majD'}, nous obtenons
 
$$\sum\limits_{n\geqslant 1}{\frac{\mu^2(n)}{n^{1+\sigma}}g(n)D(n,\boldsymbol{\chi})}\ll \frac{1}{\sigma^{\max\{y+1,m(y,\rho)\}}}\mathcal{L}_1(\sigma)^{\alpha}{\rm ,}$$
ce qui correspond à l'énoncé du Théorème \ref{theo 3} pour $D(n,\boldsymbol{\chi})$.

\section{Démonstration du Théorème \ref{theo 1}}
\subsection{Majoration de $\Delta^*(n,\chi)$}
L'objectif de cette partie est de démontrer le théorème suivant.
\begin{theoreme}
\label{theo 2}
Soient $\chi$ un caractère de Dirichlet non principal, $A$, $c$  des constantes strictement positives, $\eta \in ]0,1[$ et $g \in \mathcal{M}_A(\chi,c,\eta)$. Il existe une constante $\alpha>0$, dépendant au plus de $g$, $\chi$, $c$ et $\eta$, telle que nous ayons, uniformément pour $\sigma \in ]0, \frac{1}{10}[$
$$\sum\limits_{n\geqslant 1}{\frac{\mu^2(n)}{n^{1+\sigma}}g(n)\Delta^*(n,\chi)^2}\ll \frac{1}{\sigma^{\max\{y,m(y,\rho)-1\}}}\mathcal{L}_1(\sigma)$$
\noindent
où $y=y(g)$, $\rho$ est défini en \eqref{eq lambda}, $\mathcal{L}_1(\sigma)$ en \eqref{eq 1} et $m(y,\rho)$ en \eqref{eq m}.
\end{theoreme}

\begin{proof}
Nous commençons par utiliser la majoration \eqref{eq lemme 2.1} du Lemme~\ref{lemme 2}. Le membre de droite du majorant ne pose aucune difficulté, puisque nous avons
$$\sum\limits_{n\geqslant 1}{\frac{\mu^2(n)}{n^{1+\sigma}}g(n)\tau(n)^{1/q}} \ll \frac{1}{\sigma^{y2^{1/q}}}=\frac{1}{\sigma^{y+O(1/q)}}{\rm .}$$
\noindent
Le terme $q$ sera de l'ordre de $\sqrt{\frac{\log (1/\sigma)}{\log_2(1/\sigma)}}$, donc
$$\sum\limits_{n\geqslant 1}{\frac{\mu^2(n)}{n^{1+\sigma}}g(n)\tau(n)^{1/q}} \ll \frac{\mathcal{L}_1(\sigma)}{\sigma^y}{\rm .}$$
Nous nous intéressons donc au membre de gauche du majorant de \eqref{eq lemme 2.1}. Le Lemme \ref{lemme 4}, nous permet de considérer uniquement les entiers $n$ pour lesquels le facteur $E(n)$ n'est pas trop petit, ce qui nous amène à étudier la somme 
$$L^{\dagger}_{q}(\sigma)= \sum\limits_{n\geqslant 1}{\frac{\mu^2(n)}{n^{1+\sigma}}g(n)M^{\dagger}_{2q}(n,\chi)^{1/q}}$$
\noindent
où $M^{\dagger}_{2q}(n,\chi)$ est défini en \eqref{eq lemme 2.2}. Plus précisément nous étudions 
$$L^{\dagger}_{T,q}(\sigma)= \sum\limits_{n\geqslant 1}{\frac{\mu^2(n)}{a_nb_n^{1+\sigma}}g(n)M^{\dagger}_{2q}(n,\chi)^{1/q}}$$
\noindent
où $a_n$ et $b_n$ sont définis en \eqref{eq anbn} et $T$ un paramètre que nous fixerons par la suite.\\
Nous employons alors la méthode différentielle de \cite{B} pour majorer cette somme.\\

Pour cela, nous utilisons l'égalité
$$\Delta(np,\chi,u,v)=\Delta(n,\chi,u,v)+\chi(p)\Delta(n,\chi,u-\log p,v)$$
\noindent
valable pour tout $(n,p)=1$ si $p$ est premier. Ainsi, 
$$M^{\dagger}_{2q}(np,\chi)\leqslant 6M^{\dagger}_{2q}(n,\chi)+W^{\dagger}_{q}(n,p)$$
où
$$W^{\dagger}_{q}(n,p):=(1+q)\sum\limits_{j=1}^{q-1}{\binom{2q}{2j}N^{\dagger}_{j,q}(n,\chi,\log p)} {\rm .}$$
Nous avons donc
\begin{align*}
-\frac{{\rm d}}{{\rm d}s}L^{\dagger}_{T,q}(s)&=\sum\limits_{n \geqslant 1}{\frac{\mu^2(n)}{a_nb_n^{1+s}}g(n)M^{\dagger}_{2q}(n,\chi)^{1/q}\log b_n} \\
&\leqslant \sum\limits_{p>T}{g(p)\frac{\log p}{p^{1+s}}\sum\limits_{n\geqslant 1}{\frac{\mu^2(n)}{a_nb_n^{1+s}}g(n)M^{\dagger}_{2q}(np,\chi)^{1/q}}} \\
& \leqslant \sum\limits_{p>T}{g(p)\frac{\log p}{p^{1+s}}\sum\limits_{n\geqslant 1}{\frac{\mu^2(n)}{a_nb_n^{1+s}}g(n)\big(6M^{\dagger}_{2q}(n,\chi)+W^{\dagger}_{q}(n,p)\big)^{1/q}}}
\\
& \leqslant \sum\limits_{p>T}{g(p)\frac{\log p}{p^{1+s}}\sum\limits_{n\geqslant 1}{\frac{\mu^2(n)}{a_nb_n^{1+s}}g(n)\Big(\big(6M^{\dagger}_{2q}(n,\chi)\big)^{1/q}+W^{\dagger}_{q}(n,p)^{1/q}\Big)}}\\
& \leqslant \frac{6^{1/q}(y+as)}{s}L^{\dagger}_{T,q}(s)+\sum\limits_{n\geqslant 1}{g(n)\frac{\mu^2(n)}{a_nb_n^{1+s}}\sum\limits_{p>T}{g(p)\frac{\log p}{p^{1+s}}W^{\dagger}_{q}(n,p)^{1/q}}}{\rm .}
\end{align*}
où $a$ est une constante absolue. Une inégalité de Hölder fournit
\begin{align*}
-\frac{{\rm d}}{{\rm d}s}L^{\dagger}_{T,q}(s)\leqslant &\frac{6^{1/q}(y+as)}{s}L^{\dagger}_{T,q}(s)\\
&+\sum\limits_{n\geqslant 1}{g(n)\frac{\mu^2(n)}{a_nb_n^{1+s}}\Big(\sum\limits_{p>T}{g(p)\frac{\log p}{p}W^{\dagger}_{q}(n,p)}\Big)^{1/q}\Big(\sum\limits_{p>T}{g(p)\frac{\log p}{p^{1+s q/(q-1)}}}\Big)^{(q-1)/q}}{\rm .}
\end{align*}
L'application du Lemme \ref{lemme 6} et une inégalité de Hölder permet ainsi d'avoir
\begin{align*}
-\frac{{\rm d}}{{\rm d}s}L^{\dagger}_{T,q}(s)\leqslant &\frac{6^{1/q}(y+as)}{s}L_{T,q}(s)\\
&+\frac{A}{s^{1-1/q}}\Big(L^{\dagger}_{T,q}(s)\Big)^{(q-2)/(q-1)}\Big(\sum\limits_{n\geqslant 1}{\frac{\mu^{2}(n)}{a_nb_n^{1+s}}g(n)\widehat{\tau_{0}}(n,\chi)}\Big)^{1/(q-1)}\\
&+B_1\frac{(\log T)^{9y}}{s^{9y+1-1/q}\e^{-c/q(\log T)^{\eta}}}{\rm ,}
\end{align*}
\noindent
où $A$ et $B_1$ sont des constantes absolues. Le Lemme \ref{lemme 12} fournit une majoration du terme
$$\sum\limits_{n\geqslant 1}{\frac{\mu^{2}(n)}{a_nb_n^{1+s}}g(n)\widehat{\tau_{0}}(n,\chi)}{\rm .}$$
En effet
\begin{align*}
\widehat{\tau_{0}}(n,\chi)&=\Big(\sum\limits_{d\mid n}{\Big(\int_{\mathbb{R}}{|\tau(d,\chi,\vartheta)|^2\frac{{\rm d} \vartheta}{1+\vartheta^2}}\Big)^q}\Big)^{1/q} \\
&\leqslant \tau(n)^{1/q}\max\limits_{d \mid n}\Big(\int_{\mathbb{R}}{|\tau(d,\chi,\vartheta)|^2\frac{{\rm d} \vartheta}{1+\vartheta^2}}\Big) \\
& \leqslant \tau(n)^{1/q}\int_{\mathbb{R}}{\max\limits_{d \mid n}\big(|\tau(d,\chi,\vartheta)|^2\big)\frac{{\rm d} \vartheta}{1+\vartheta^2}}{\rm ,}
\end{align*}
ce qui nous permet d'utiliser ce lemme. Nous obtenons donc
\begin{align*}
-\frac{{\rm d}}{{\rm d}s}L^{\dagger}_{T,q}(s)\leqslant &\frac{6^{1/q}(y+as)}{s}L_{T,q}(s)+A\Big(L_{T,q}(s)\Big)^{(q-2)/(q-1)}\frac{1}{s^{1+m(y,\rho)/(q-1)-1/q+O(1/q^2)}}\\
&+B_1\frac{(\log T)^{9y}}{s^{9y+1-1/q}\e^{-c/q(\log T)^{\eta}}}{\rm ,}
\end{align*}
où $m(y,\rho)$ est défini en \eqref{eq m}.
\noindent
En posant 
$$\varepsilon:=B_1(\log T)^{9y}e^{-(\log T)^{\eta}c/q}, \; \; \beta:=9y-1/q {\rm ,}$$
\noindent 
nous obtenons
$$-{L^{\dagger}}'_{T,q}(s)\leqslant \phi\big(s,L^{\dagger}_{T,q}(s)\big){\rm ,}$$
\noindent
avec
$$\phi(s,x):=6^{1/q}\frac{y+as}{s}x+\frac{Ax^{(q-2)/(q-1)}}{s^{1-1/q+m(y,\rho)/(q-1)}}+\frac{\varepsilon}{s^{\beta+1}} {\rm .}$$
\noindent
Posons 
$$\gamma:=\max\Big(m(y,\rho)-\frac{q-1}{q}+O\Big(\frac{q-1}{q^2}\Big),y+\frac{b}{q}\Big), \; \; X(s):=\frac{K}{s^{\gamma}}+\frac{\varepsilon}{s^{\beta}}$$
\noindent
pour une certaine constante $b$, de sorte que 
$$\gamma+1\geqslant 1-\frac{1}{q}+\frac{m(y,\rho)}{q-1}+O(\frac{1}{q^2})+\frac{q-2}{q-1}\gamma {\rm ,}$$
et
$$\gamma> 6^{1/q}(y+as_0){\rm ,}$$
où $s_0$ est petit devant $\frac{1}{q}$.
\noindent
Pour $s_0$ suffisamment petit devant $\frac{1}{q}$, en supposant les inégalités
\begin{align*}
&\gamma>6^{1/q}(y+as_0)+AK^{-1/(q-1)}{\rm ,} \\
&\beta>6^{1/q}(y+as_0)+AK^{-1/(q-1)}+1{\rm ,}
\end{align*}
\noindent
nous avons
\begin{align*}
-X'(s)=&\frac{K\gamma}{s^{\gamma+1}}+\frac{\epsilon\beta}{s^{\beta+1}} \\
\geqslant& \frac{6^{1/q}K(y+as)}{s^{\gamma+1}}+\frac{\varepsilon6^{1/q}(y+as)}{s^{\beta+1}}+\frac{AK^{(q-2)/(q-1)}}{s^{\gamma+1}}\Big(1+\frac{\varepsilon s^{\gamma}}{Ks^{\beta}}\Big)\\
&+\frac{\varepsilon}{s^{\beta+1}} \\
\geqslant& \frac{6^{1/q}K(y+as)}{s^{\gamma+1}}+\frac{\varepsilon6^{1/q}(y+as)}{s^{\beta+1}}\\
&+\frac{AK^{(q-2)/(q-1)}}{s^{\gamma+1}}\Big(1+\frac{\varepsilon s^{\gamma}}{Ks^{\beta}}\Big)^{(q-2)/(q-1)}+\frac{\varepsilon}{s^{\beta+1}}\\
\geqslant &\phi\big(s,X(s)\big){\rm .}
\end{align*}

\begin{remarque}
Les hypothèses sur $\gamma$ et $\beta$ sont vérifiées pour $s_0$ suffisamment petit devant $\frac{1}{q}$ dès que nous choisissons $K$ de l'ordre de $(qC_2)^q$ pour une constante $C_2$ assez grande.
\end{remarque}

Comme tout ceci n'est valable que pour $s\leqslant \frac{1}{\log T}$, nous notons $s_0=\frac{1}{\log T}$. le réel $T$ sera choisi de telle sorte que $\frac{1}{\log T}$ soit petit devant $\frac{1}{q}$. Remarquons dans ce cas que
$$L^{\dagger}_{T,q}(s)\leqslant \sum\limits_{n\geqslant 1}{\frac{\mu^2(n)}{a_nb_n^{1+s}}\tau_3(n)^{2}} \ll \frac{(\log T)^{9y}}{s^{9y}} {\rm .} $$ 
Ainsi
$$L^{\dagger}_{T,q}(s_0)\leqslant X(s_0)$$
\noindent
en choisissant $K=(qC_2)^q(\log T)^{18}$. Le lemme 70.2 de \cite{HT} permet alors de montrer l'inégalité suivante
$$L^{\dagger}_{T,q}(s)\leqslant \frac{(qC_2)^q(\log T)^{18}}{s^{m(y,\rho)-1+1/q}}+\frac{B(\log T)^{9y}}{\e^{(\log T)^{\eta}c/q}s^{9y-1/q}} {\rm .}$$
\noindent
Nous choisissons $T$ de sorte que 
$$\e^{(\log T)^{\eta}c/q}=\sigma^{-9y}$$
\noindent
ce qui implique 
$$\Big(\frac{9yq\log \frac{1}{\sigma}}{c}\Big)^{1/\eta}=\log T {\rm .}$$
\noindent
Dans le cas où $q$ est de l'ordre de $\sqrt{\frac{\log(1/\sigma)}{\log_2(1/\sigma)}}$, nous avons bien $\sigma\leqslant \frac{1}{\log T}$. De plus, le terme dominant est celui de gauche, ainsi 
$$L^{\dagger}_{T,q}(\sigma)\ll \frac{1}{\sigma^{\max\{y,m(y,\rho)-1\}}}\mathcal{L}_1(\sigma)^{\alpha}{\rm ,} $$
où $\alpha$ est une constante strictement positive.
\end{proof}
\subsection{Fin de la démonstration}
Dans cette partie, nous utilisons les différents théorèmes et lemmes vus précédemment afin de démontrer le Théorème \ref{theo 1}. Pour commencer nous appliquons la méthode de Rankin, ainsi que la formule (25) de \cite{T} pour obtenir la majoration
$$\mathfrak{S}(x,g,\boldsymbol{\chi})\ll \frac{x^{1+\sigma}}{\log x}\sum\limits_{n\geqslant 1}{\frac{\mu^2(n)}{n^{1+\sigma}}g(n)\Delta_3(n,\boldsymbol{\chi})^2}:=\mathfrak{S}_0(\sigma,g,\boldsymbol{\chi}){\rm .}$$
Le Lemme \ref{lemme 1} permet de majorer $\Delta_3(n,\boldsymbol{\chi})^2$. La contribution des deuxième et troisième termes du majorant est, selon le Théorème \ref{theo 2} $\ll \frac{1}{\sigma^{\max\{y,m(y,\rho)-1\}}}\mathcal{L}_1(\sigma)^{\alpha}$.
Pour traiter le premier terme du majorant, nous utilisons le Lemme \ref{lemme 4} pour ne travailler que sur la somme suivante
$$L_{q}(\sigma):= \sum\limits_{n\geqslant 1}{\frac{\mu^2(n)}{n^{1+\sigma}}g(n)M_{2q}(n,\boldsymbol{\chi})^{1/q}}{\rm ,}$$
\noindent
où $M_{q}(n,\boldsymbol{\chi})$ est défini en \eqref{eq M}. Plus précisément, nous étudions 
$$L_{T,q}(\sigma):= \sum\limits_{n\geqslant 1}{\frac{\mu^2(n)g(n)}{a_nb_n^{1+\sigma}}M_{2q}(n,\boldsymbol{\chi})^{1/q}} {\rm ,}$$
où $a_n$ et $b_n$ sont définis en \eqref{eq anbn}.
En dérivant par rapport à $\sigma$ la fonction précédente, nous obtenons
\begin{align*}
-L'_{T,q}(\sigma)&=\sum\limits_{n\geqslant 1}{\frac{\mu^2(n)}{a_nb_n^{1+\sigma}}g(n)M_{2q}(n,\boldsymbol{\chi})^{1/q}\log b_n}\\
&=\sum\limits_{n\geqslant 1}{\frac{\mu^2(n)}{a_nb_n^{1+\sigma}}g(n)M_{2q}(n,\boldsymbol{\chi})^{1/q}\sum\limits_{\substack{p\mid n \\ p>T}}{\log p}}\\
&=\sum\limits_{n\geqslant 1}{\frac{\mu^2(n)}{a_nb_n^{1+\sigma}}g(n)\sum\limits_{ p>T}{\frac{g(p)}{p^{1+\sigma}}\log p M_{2q}(np,\boldsymbol{\chi})^{1/q}}}{\rm .}
\end{align*}
Pour majorer $M_{2q}(np,\boldsymbol{\chi})$, il nous faut dans un premier temps calculer $\Delta_3(np,\boldsymbol{\chi},\mathbf{u},\mathbf{v})$. Nous avons
\begin{align}
\begin{split}
\Delta_3(np,\boldsymbol{\chi},\mathbf{u},\mathbf{v})=&\Delta_3(n,\boldsymbol{\chi},\mathbf{u},\mathbf{v})+\chi_1(p)\Delta_3(n,\boldsymbol{\chi},u_1-\log p,u_2,\mathbf{v})\\
&+\chi_2(p)\Delta_3(n,\boldsymbol{\chi},u_1,u_2,\mathbf{v}){\rm .}
\end{split}
\end{align}
En prenant le module à la puissance $2q$, nous obtenons
\begin{align}
\begin{split}
&\Big\lvert\Delta_3(np,\boldsymbol{\chi},\mathbf{u},\mathbf{v})\Big\rvert^{2q}\\
\leqslant& \sum\limits_{i+j+k=2q}{\binom{2q}{i,j,k}\Big\lvert\Delta_3(n,\boldsymbol{\chi},\mathbf{u},\mathbf{v})\Big\rvert^{i}\Big\lvert\Delta_3(n,\boldsymbol{\chi},u_1-\log p,u_2,\mathbf{v})\Big\rvert^{j}\Big\lvert\Delta_3(n,\boldsymbol{\chi},u_1,u_2-\log p,\mathbf{v})\Big\rvert^{k}}
\end{split}
\end{align}
Si les trois exposants $i,j,k$ sont non nuls, nous considérons le plus grand de ces trois exposants. Nous pouvons supposer sans perdre de généralité qu'il s'agit de $k$. Nous décomposons alors $k$ de la manière suivante 
$$k=k_1+k_2{\rm ,}$$
où $k_1$ et $k_2$ vérifient
\begin{align*}
i+k_1=q\\
j+k_2=q{\rm .}
\end{align*}
Nous appliquons alors l'inégalité 
$$ab\leqslant \frac{1}{2}(a^2+b^2)$$
à
$$a=\Big\lvert\Delta_3(n,\boldsymbol{\chi},\mathbf{u},\mathbf{v})\Big\rvert^{i}\Big\lvert\Delta_3(n,\boldsymbol{\chi},u_1,u_2-\log p,\mathbf{v})\Big\rvert^{k_1}$$
et 
$$b=\Big\lvert\Delta_3(n,\boldsymbol{\chi},u_1-\log p,u_2,\mathbf{v})\Big\rvert^{j}\Big\lvert\Delta_3(n,\boldsymbol{\chi},u_1,u_2-\log p,\mathbf{v})\Big\rvert^{k_2}$$
afin d'obtenir
\begin{align}
\begin{split}
&\Big\lvert\Delta_3(n,\boldsymbol{\chi},\mathbf{u},\mathbf{v})\Big\rvert^{i}\Big\lvert\Delta_3(n,\boldsymbol{\chi},u_1-\log p,u_2,\mathbf{v})\Big\rvert^{j}\Big\lvert\Delta_3(n,\boldsymbol{\chi},u_1,u_2-\log p,\mathbf{v})\Big\rvert^{k}\\
\leqslant& \frac{1}{2}\Big\lvert\Delta_3(n,\boldsymbol{\chi},\mathbf{u},\mathbf{v})\Big\rvert^{2i}\Big\lvert\Delta_3(n,\boldsymbol{\chi},u_1,u_2-\log p,\mathbf{v})\Big\rvert^{2(q-i)}\\
&+\frac{1}{2}\Big\lvert\Delta_3(n,\boldsymbol{\chi},u_1-\log p,u_2,\mathbf{v})\Big\rvert^{2j}\Big\lvert\Delta_3(n,\boldsymbol{\chi},u_1,u_2-\log p,\mathbf{v})\Big\rvert^{2(q-j)}{\rm .}
\end{split}
\end{align}
Si l'un des trois exposants est nul, disons $k$, nous utilisons les majorations suivantes.
\begin{align*}
&\Big\lvert\Delta_3(n,\boldsymbol{\chi},\mathbf{u},\mathbf{v})\Big\rvert^{2h+1}\Big\lvert\Delta_3(n,\boldsymbol{\chi},u_1,u_2-\log p,\mathbf{v})\Big\rvert^{2q-2h-1}\\
\leqslant& \frac{1}{2}\Big\lvert\Delta_3(n,\boldsymbol{\chi},\mathbf{u},\mathbf{v})\Big\rvert^{2h}\Big\lvert\Delta_3(n,\boldsymbol{\chi},u_1,u_2-\log p,\mathbf{v})\Big\rvert^{2q-2h}\\
&+\frac{1}{2}\Big\lvert\Delta_3(n,\boldsymbol{\chi},\mathbf{u},\mathbf{v})\Big\rvert^{2h+2}\Big\lvert\Delta_3(n,\boldsymbol{\chi},u_1,u_2-\log p,\mathbf{v})\Big\rvert^{2q-2h-2} && (1\leqslant h\leqslant q-2){\rm ,}\\
&\Big\lvert\Delta_3(n,\boldsymbol{\chi},\mathbf{u},\mathbf{v})\Big\rvert\Big\lvert\Delta_3(n,\boldsymbol{\chi},u_1,u_2-\log p,\mathbf{v})\Big\rvert^{2q-1}\\
\leqslant& \frac{1}{2q}\Big\lvert\Delta_3(n,\boldsymbol{\chi},u_1,u_2-\log p,\mathbf{v})\Big\rvert^{2q}\\
&+\frac{q}{2}\Big\lvert\Delta_3(n,\boldsymbol{\chi},\mathbf{u},\mathbf{v})\Big\rvert^{2}\Big\lvert\Delta_3(n,\boldsymbol{\chi},u_1,u_2-\log p,\mathbf{v})\Big\rvert^{2q-2} {\rm ,}\\
&\Big\lvert\Delta_3(n,\boldsymbol{\chi},\mathbf{u},\mathbf{v})\Big\rvert^{2q-1}\Big\lvert\Delta_3(n,\boldsymbol{\chi},u_1,u_2-\log p,\mathbf{v})\Big\rvert\\
\leqslant & \frac{q}{2}\Big\lvert\Delta_3(n,\boldsymbol{\chi},\mathbf{u},\mathbf{v})\Big\rvert^{2(q-1)}\Big\lvert\Delta_3(n,\boldsymbol{\chi},u_1,u_2-\log p,\mathbf{v})\Big\rvert^{2}\\
&+\frac{1}{2q}\Big\lvert\Delta_3(n,\boldsymbol{\chi},\mathbf{u},\mathbf{v})\Big\rvert^{2q}{\rm .}
\end{align*}
La majoration de  $M_{2q}(np,\boldsymbol{\chi})$ fait donc intervenir les quantités $N_{1,j,q}(n,\boldsymbol{\chi},\log p)$ et $N_{2,j,q}(n,\boldsymbol{\chi},\log p)$ définies respectivement en \eqref{N1jq} et \eqref{N2jq}. Pour la suite, nous définissons
\begin{equation}
\label{eq n}
n(y,\rho):=\max\{y+1,(\rho+2)y,3y-1\}=\max\{y+1,m(y,\rho)\}{\rm .}
\end{equation}
Le Lemme \ref{lemme 7} fournit alors
\begin{align*}
-L'_{T,q}(s)\leqslant &\frac{9^{1/q}(y+as)}{s}L_{T,q}(s)\\
&+\frac{A}{s^{1-1/q}}\Big(L_{T,q}(s)\Big)^{(q-2)/(q-1)}\Big(\sum\limits_{n\geqslant 1}{\frac{\mu^2(n)g(n)}{a_nb_n^{1+s}}C(n,\boldsymbol{\chi})}\Big)^{1/q}\\
&\Big(\sum\limits_{n \geqslant 1}{\frac{\mu^2(n)g(n)}{a_nb_n^{1+s}}\widehat{\tau_{2}}(n,\boldsymbol{\chi})}\Big)^{1/q(q-1)}\\
&+\frac{A}{s^{1-1/q}}\Big(L_{T,q}(s)\Big)^{(q-2)/(q-1)}\Big(\sum\limits_{n\geqslant 1}{\frac{\mu^2(n)g(n)}{a_nb_n^{1+s}}C(n,^t\boldsymbol{\chi})}\Big)^{1/q}\\
&\Big(\sum\limits_{n \geqslant 1}{\frac{\mu^2(n)g(n)}{a_nb_n^{1+s}}\widehat{\tau_{2}}(n,\boldsymbol{\chi})}\Big)^{1/q(q-1)}\\
&+\frac{A}{s^{1-1/q}}\Big(L_{T,q}(s)\Big)^{(q-2)/(q-1)}\Big(\sum\limits_{n\geqslant 1}{\frac{\mu^2(n)g(n)}{a_nb_n^{1+s}}D(n,\boldsymbol{\chi})}\Big)^{1/q}\\
&\Big(\sum\limits_{n \geqslant 1}{\frac{\mu^2(n)g(n)}{a_nb_n^{1+s}}\widehat{\tau_{2}}(n,\boldsymbol{\chi})}\Big)^{1/q(q-1)}\\
&+B_1\frac{(\log T)^{9y}}{s^{9y+1-1/q}\e^{-c/q(\log T)^{\eta}}}
\end{align*}
\noindent
où $^t\boldsymbol{\chi}=(\chi_2, \chi_1)$.
\noindent
D'après le Théorème \ref{theo 3}, nous avons
\begin{align*}
-L'_{T,q}(s)\leqslant &\frac{9^{1/q}(y+as)}{s}L_{T,q}(s)+A\Big(L_{T,q}(s)\Big)^{(q-2)/(q-1)}\frac{\mathcal{L}_1(s)^{\alpha/(q-1)}}{s^{1+(n(y,\rho)/q-1/q+9y/q(q-1)}}\\
&+B_1\frac{(\log T)^{9y}}{s^{9y+1-1/q}\e^{-c/q(\log T)^{\eta}}}{\rm .}
\end{align*}
\noindent
En posant 
$$\varepsilon:=B_1(\log T)^{9y}e^{-(\log T)^{\eta}c/q}, \; \; \beta:=9y-1/q{\rm ,}$$
\noindent 
nous obtenons
$$-L'_{T,q}(s)\leqslant \phi_2\big(s,L_{T,q}(s)\big){\rm ,}$$
\noindent
avec
$$\phi_2(s,x):=9^{1/q}\frac{y+as}{s}x+\frac{A\mathcal{L}_1(\sigma)^{\alpha/(q-1)}x^{(q-2)/(q-1)}}{s^{1-1/q+n(y,\rho)/(q-1)+9y/q(q-1)}}+\frac{\varepsilon}{s^{\beta+1}} {\rm .}$$
\noindent
Posons 
$$\gamma_2:=\max\Big(y+\frac{b}{q},n(y,\rho)-1+\frac{9y}{q}\Big), \; \; X_2(s):=\frac{K_2}{s^{\gamma_2}}+\frac{\varepsilon}{s^{\beta}}{\rm ,}$$
\noindent
de sorte que nous ayons les inégalités suivantes,
$$\gamma_2+1\geqslant 1-\frac{1}{q}+\frac{n(y,\rho)}{q-1}+\frac{9y}{q(q-1)}+\frac{q-2}{q-1}\gamma_2 {\rm ,}$$
et 
$$\gamma_2>9^{1/q}(y+as_0)$$
pour $s_0$ suffisamment petit devant $\frac{1}{q}$. Pour que $X_2$ vérifie
$$-X_2'(s)\geqslant \phi_2\big(s,X_2(s)\big){\rm ,}$$
il faut choisir $K_2$ de l'ordre de $(qC_3)^q\mathcal{L}_1(\sigma)^{\alpha}$ où $C_3$ est une constante suffisamment grande. Nous obtenons finalement, pour les choix de $T$ et de $q$ définis en \eqref{eq T} et \eqref{eq q},
$$L_{T,q}(\sigma)\ll\frac{1}{\sigma^{n(y,\rho)-1}}\mathcal{L}_1(\sigma)^{\alpha_1}{\rm ,}$$
où $\alpha_1>\alpha$.
En spécifiant $\sigma=1/\log x$, nous obtenons, d'après \eqref{eq n},
$$\mathfrak{S}(x,g,\boldsymbol{\chi})\ll x(\log x)^{\max\{y-1,(\rho+2)y-2,3y-3\}}\mathcal{L}(x)^{\alpha_1}{\rm .}$$

\section{Démonstration du Théorème \ref{prop 2}}

\begin{proof}
L'inégalité 
$$\sum\limits_{n\geqslant x}{\mu^2(n)y^{\omega(n)}\Delta_3(n,\boldsymbol{\chi})^2}\gg x(\log x)^{y-1}$$
se déduit de la minoration triviale
$$\Delta_3(n,\boldsymbol{\chi})^2 \geqslant 1{\rm .}$$

La deuxième inégalité est obtenue par application du théorème de Parseval.
Comme la fonction $\Delta_3(n,\boldsymbol{\chi},\mathbf{u},\mathbf{v})$ est nulle si $u_1$ ou $u_2$ est plus grand que $\log n$ en valeurs absolue, nous pouvons écrire que pour tout $n\leqslant x$

\begin{equation}
\label{eq 44}
\Delta_3^2(n,\boldsymbol{\chi})\gg\frac{1}{(\log x)^2} \int_{\mathbb{R}^2 \times [0,1]^2}{|\Delta_3(n,\boldsymbol{\chi},\mathbf{u},\mathbf{v})|^2{\rm d}\mathbf{u}{\rm d}\mathbf{v}} {\rm .}
\end{equation}
La fonction $\mathbf{u}\rightarrow \Delta_3(n,\boldsymbol{\chi},\mathbf{u},\mathbf{v})$ est  intégrable et de carré intégrable, nous pouvons donc utiliser le théorème de Parseval pour calculer l'intégrale ci-dessus. Nous obtenons 
$$\int_{\mathbb{R}^2 \times [0,1]^2}{|\Delta_3(n,\boldsymbol{\chi},\mathbf{u},\mathbf{v})|^2{\rm d}\mathbf{u}{\rm d}\mathbf{v}}=\int_{\mathbb{R}^2 \times [0,1]^2}{|\tau(n,\boldsymbol{\chi},\boldsymbol{\vartheta})|^2\prod\limits_{i=1}^{2}{\frac{\sin^2(v_i\vartheta_i/2)}{(\vartheta_i/2)^2}}{\rm d}\boldsymbol{\vartheta}{\rm d}\mathbf{v}} {\rm ,}$$
où $\tau(n,\boldsymbol{\chi},\boldsymbol{\vartheta})$ est défini en \eqref{tautheta12}.

L'avantage de cette nouvelle fonction est qu'elle est multiplicative, nous savons donc estimer sa fonction sommatoire. Tout d'abord le calcul de l'intégrale en $v_1$ et $v_2$ montre que 

\begin{equation}
\label{eq 2}
\int_{\mathbb{R}^2 \times [0,1]^2}{|\Delta_3(n,\boldsymbol{\chi},\mathbf{u},\mathbf{v})|^2{\rm d}\mathbf{u}{\rm d}\mathbf{v}} \gg \int_{[0,1]^2}{|\tau(n,\boldsymbol{\chi},\boldsymbol{\vartheta})|^2{\rm d}\boldsymbol{\vartheta}} {\rm .}
\end{equation}
Soit $F(s)$ la série de Dirichlet associée à la fonction
$$\mu^2(n)y^{\omega(n)}|\tau(n,\boldsymbol{\chi},\boldsymbol{\vartheta})|^2 {\rm .}$$
Pour $\mathfrak{Re}(s)>1$, cette série de Dirichlet admet le produit eulérien suivant
$$F(s)=\prod\limits_{p}{\Big(1+\frac{y}{p^s}|1+\chi_1(p)p^{i\vartheta_1}+\chi_2(p)p^{i\vartheta_2}|^2\Big)} $$
que nous pouvons réécrire sous la forme
\begin{align*}
F(s)= & \prod\limits_{p}{\Big\{1+\frac{y}{p^s}\Big(3 +\chi_1(p)p^{i\vartheta_1}+\overline{\chi_1(p)}p^{-i\vartheta_1}+\chi_2(p)p^{i\vartheta_2}+\overline{\chi_2(p)}p^{-i\vartheta_2}} \\
& +\chi_1(p)\overline{\chi_2(p)}p^{i(\vartheta_1-\vartheta_2)}+\overline{\chi_1(p)}\chi_2(p)p^{i(\vartheta_2-\vartheta_1)}\Big)\Big\} {\rm .}
\end{align*}

\noindent
Dans le cas où $\chi_1\overline{\chi_2}$ n'est pas principal, nous avons 

\begin{align}
\begin{split}
\label{eq 3}
F(s)= & \zeta(s)^{3y}L(s-i\vartheta_1,\chi_1)^yL(s+i\vartheta_1,\overline{\chi_1})^yL(s-i\vartheta_2,\chi_2)^yL(s+i\vartheta_2,\overline{\chi_2})^y\\
& L(s-i(\vartheta_1-\vartheta_2),\chi_1\overline{\chi_2})^yL(s+i(\vartheta_1-\vartheta_2),\overline{\chi_1}\chi_2)^yF_1(s){\rm ,}
\end{split}
\end{align}
où $F_1(s)$ est une série de Dirichlet absolument  convergente pour $\Re e(s)>3/4$. \\

La méthode de Selberg-Delange assure alors que lorsque $|\vartheta_1|$, $|\vartheta_2|$ et $|\vartheta_1-\vartheta_2|$ sont plus grands que $c/\log(x)$ pour une certaine constante $c$, nous avons

\begin{equation}
\label{eq 4}
\sum\limits_{n\leqslant x}{\mu^2(n)y^{\omega(n)}|\tau(n,\boldsymbol{\chi},\boldsymbol{\vartheta})|^2}\gg x\log(x)^{3y-1} {\rm .}
\end{equation}

\noindent
Les formules \eqref{eq 44},\eqref{eq 2} et \eqref{eq 4} permettent de conclure dans le cas où $\chi_1\overline{\chi_2}$ est non principal. \\

Dans le cas où $\chi_1=\chi_2=\chi$, les deux dernières fonctions $L$ qui apparaissaient dans la formule \eqref{eq 3} ne sont plus des fonctions $L$ mais des fonctions $\zeta$, cependant le fait que leur argument ne soit pas $s$ nous empêche d'appliquer la méthode de Selberg-Delange pour conclure. Afin de contourner cette difficulté, nous constatons que si $\vartheta_1$ et $\vartheta_2$ sont très proches, alors $\zeta\big(s+i(\vartheta_1-\vartheta_2)\big)$ est proche de  $\zeta(s)$, ce qui nous conduit à énoncer le lemme suivant.

\begin{lemme}

Pour $x\geqslant 16$, il existe une constante $c$ telle que pour $|\vartheta_1-\vartheta_2|\leqslant c/\log x$, il existe une constante $B$ dépendant uniquement de $c$ telle que pour tout $n\leqslant x$
sans facteur carré
$$\Big|\log\Big(\frac{f(n)}{g(n)}\Big)\Big|\leqslant B$$

\noindent
où

$$f(n)=|\tau(n,\chi,\chi,\vartheta_1,\vartheta_2)|^2$$

\noindent
et 

$$g(n)=|\tau(n,\chi,\chi,\vartheta_1,\vartheta_1)|^2{\rm .}$$
\end{lemme}

Supposons le lemme démontré. Dans ce cas sur la bande $|\vartheta_1-\vartheta_2|\leqslant c/\log x$, nous avons

\begin{align}
\label{eq 5}
\sum\limits_{n\leqslant x}{\mu^2(n)y^{\omega(n)}f(n)}\gg \sum\limits_{n\leqslant x}{\mu^2(n)y^{\omega(n)}g(n)}
\end{align}

\noindent
et si nous notons $G$ la série de Dirichlet associée à $\mu^2(n)y^{\omega(n)}g(n)$, nous pouvons écrire
$$G(s)=\zeta(s)^{5y}L(s-i\vartheta_1,\chi)^{2y}L(s+i\vartheta_1,\chi)^{2y}G_1(s){\rm ,}$$
où $G_1(s)$ est une série de Dirichlet absolument convergente pour $\Re e(s)>3/4$. \\

La méthode de Selberg-Delange assure alors que si $|\vartheta_1|\geqslant c'/\log x$ pour une certaine constante $c'>0$, alors

\begin{align}
\label{eq 6}
\sum\limits_{n\leqslant x}{\mu^2(n)y^{\omega(n)}g(n)}\gg x\log(x)^{5y-1} {\rm .}
\end{align}
Les formules \eqref{eq 1}, \eqref{eq 2}, \eqref{eq 5} (qui n'est valable que sur un ensemble de mesure $1/\log x$) et \eqref{eq 6} permettent de conclure dans le cas où $\chi_1=\chi_2$.
\end{proof}

Il nous reste à démontrer le lemme.

\begin{proof}

Pour $p\leqslant x$, nous avons

$$\frac{\tau(p,\chi,\chi,\vartheta_1,\vartheta_2)}{\tau(p,\chi,\chi,\vartheta_1,\vartheta_1)}=\frac{1+\chi(p)\big(p^{i\vartheta_1}+p^{i\vartheta_2}\big)}{1+2\chi(p)p^{i\vartheta_1}}{\rm .}$$
Par inégalité triangulaire, le module du dénominateur est toujours plus grand que $1$, nous pouvons donc écrire
\begin{align*}
\frac{\tau(p,\chi,\chi,\vartheta_1,\vartheta_2)}{\tau(p,\chi,\chi,\vartheta_1,\vartheta_1)}& =\frac{1+2\chi(p)p^{i\vartheta_1}+\chi(p)\big(p^{i\vartheta_2}-p^{i\vartheta_1}\big)}{1+2\chi(p)p^{i\vartheta_1}}\\
&=1+O\big((\vartheta_1-\vartheta_2)\log(p)\big) {\rm .}
\end{align*}

\noindent
Ainsi, si la différence $\vartheta_1-\vartheta_2$ est suffisamment petite par rapport à $1/\log x$, d'une part la fraction ne s'annule pas, nous pouvons donc considérer le logarithme du module de cette fraction, et d'autre part, le développement limité du logarithme nous assure que 

$$\log\Big|\frac{\tau(p,\chi,\chi,\vartheta_1,\vartheta_2)}{\tau(p,\chi,\chi,\vartheta_1,\vartheta_1)}\Big|=O\big((\vartheta_1-\vartheta_2)\log(p)\big) {\rm .}$$

\noindent
Nous avons donc pour tout $n\leqslant x$ sans facteur carré

\begin{align*}
\log\Big(\frac{f(n)}{g(n)}\Big) & =\sum\limits_{p\mid n}{2\log\Big|\frac{\tau(p,\chi,\vartheta_1,\vartheta_2)}{\tau(p,\chi,\vartheta_1,\vartheta_1)}\Big|} \\
&=\sum\limits_{p\mid n}{O\big((\vartheta_1-\vartheta_2)\log(p)\big)} \\
&=O\big((\vartheta_1-\vartheta_2)\log(n)\big)\\
&=O(1){\rm .}
\end{align*}

\noindent
D'où le résultat.
\end{proof}

\section*{Remerciements}

Je tiens ici à remercier mon directeur de thèse, Régis de la Bretèche, et Gérald Tenenbaum, qui, au travers de nombreuses discussions, m'ont permis d'obtenir des avancées significatives dans la résolution de mon problème, et qui ont également pris le temps de répondre aux nombreuses questions que je leur ai posées.

\nocite{*}
\bibliographystyle{amsplain}


\bibliography{biblio4}
\addcontentsline{toc}{section}{Références}

\end{document}